\documentclass[reqno,11pt,a4paper,amsmath,amssymb]{amsart}
\usepackage{amsmath,amssymb,amsthm,graphicx,mathrsfs,url}
\usepackage[utf8]{inputenc}
\usepackage[usenames,dvipsnames]{color}
\usepackage[colorlinks=true,linkcolor=red,citecolor=green]{hyperref}
\usepackage{enumitem}
\usepackage{csquotes}
\usepackage[all]{xy}
\usepackage{amsfonts}
\usepackage{dsfont}
\usepackage{xcolor}
\usepackage{array}
\usepackage[margin=1.4in]{geometry}
\usepackage[nothm]{thmbox}
\usepackage{mdframed}
\usepackage{graphicx,import}
\usepackage{subcaption}
\usepackage{quiver}
\usepackage{tikz-cd}

\newcommand{\R}{\mathbb{R}}

\newcommand{\C}{\mathbb{C}}

\newcommand{\Hyp}{\mathcal{H}}
\newcommand{\Ell}{\mathcal{E}}
\DeclareMathOperator{\SL}{SL}

\newcommand{\cv}{\mathcal{X}}

\newcommand{\spe}{\mathrm{Sp}}
\newcommand{\Ts}{\mathrm{T}^{*} }

\newtheoremstyle{MyTheorem}
{} 
{} 
{\itshape} 
{} 
{\scshape} 
{.} 
{\newline} 
{} 

\newcommand\Z{\mathbb{Z}}

\newcommand\tr{\mathrm{tr}}

\usepackage{mathtools}
\DeclarePairedDelimiter\abs{\lvert}{\rvert}
\DeclarePairedDelimiter\norm{\lVert}{\rVert}

\usepackage{dirtytalk}

\title{Spectrum of $\mathrm{SL}(2,\mathbb{R})$-characters: the once-punctured torus case}
\author{Florestan Martin-Baillon}
\author{Selim Ghazouani}
\address{}
\thanks{{\footnotesize{The research activities of F.M.B. are partially
funded by the European Research Council (ERC PosLieRep 101018839).}}}

\numberwithin{equation}{section}

\theoremstyle{definition}

\newtheorem{definition}[equation]{Definition}

\theoremstyle{theorem}

\newtheorem{theorem}[equation]{Theorem}
\newtheorem{proposition}[equation]{Proposition}
\newtheorem{corollary}[equation]{Corollary}
\newtheorem{conjecture}[equation]{Conjecture}
\newtheorem{lemma}[equation]{Lemma}

\newtheorem{question}[equation]{Question}

\begin{document}

\date{\today}
\begin{abstract}
	Consider a topological surface $\Sigma$.
	We introduce the \emph{spectrum} of
	representation $\pi_1 \Sigma \longrightarrow \mathrm{SL}(2, \mathbb{R})$,
	which is a subset of projective measured lamination on the surface,
	which captures the directions along which the representation
	fails to be Fuchsian,
	and which characterizes the action of the mapping class group
	on this representation.

	In the case of the once-punctured torus,
	we show that the spectrum of a generic representation is a Cantor set,
	and that it completely describes the dynamics of
	the familly of
	\emph{locally constant cocycles} above interval exchange transformations
	associated to the representation.
\end{abstract}
\maketitle

\tableofcontents

\section*{Introduction}

Let $ \Sigma = \Sigma_{g,n}$ be the compact, oriented surface of genus $g$ with
$n$ marked points. The aim of the present article is to introduce a new object
associated to a representation of a surface group $\rho : \pi_1 \Sigma
\longrightarrow \mathrm{SL}(2,\mathbb{R})$, which we call its \textit{spectrum}.
It is a subset of the set of measured laminations on $\Sigma$ which essentially
records the locus of simple closed curves (thought of as elements in $\pi_1
\Sigma$) where $\rho$ fails to
``look like a Fuchsian representation''.

Our motivation comes from the study of the dynamics of the mapping class group
on spaces of representations.
It is well known that the action of the mapping class group
on the set of Fuchsian (more generally quasi-Fuchsian) representations
is properly discontinuous.
A long-standing problem is to understand the space of
non-Fuchsian representations,
and the action of the mapping class group on this space.
Since the work of Minsky~\cite{minskyDynamicsOutFnPSL2C2013}
for free groups,
adapted to surface groups by~\cite{TholozanWang},
we know that the action of the mapping class group is still properly
discontinuous on the larger set of \emph{simple Fuchsian representations}:
by definition, they are representations that behave like a Fuchsian
representation in restriction to simple geodesics.
This motivates the following definition.

\vspace{2mm}

\paragraph*{\it Loose definition of the spectrum}
The particular feature of simple Fuchsian representations that we are going to focus on is the following:
if $\gamma$ is a class of conjugacy corresponding to a simple curve,
of length $l > 0$ in $\pi_1 \Sigma$ (measured for some arbitrary word metric), the spectral radius of $\rho(\gamma)$ is of the order at least $C^l$ for some constant $C > 1$. 

\vspace{2mm} For a non-simple Fuchsian representation $\rho$,
this property fails to be true. Precisely, there exists a sequence $(\gamma_n)$ of conjugacy classes of elements of $\pi_1 \Sigma$ such that the spectral radius of $\rho(\gamma_n)$ grows sub-exponentially in the length of $\gamma_n$, measured with a word metric still. 
The space of free homotopy classes of simple  closed curves can be
compactified in a natural way using \textit{measured laminations} on $\Sigma$. We thus define the spectrum of a representation as 
\[ \spe(\rho) \subset \mathcal{ML}\]
the set of measured laminations which can be approached by a sequence of simple closed curves $(\gamma_n)$ along which the spectral radius of $\rho(\gamma_n)$ grows subexponentially fast (see paragraph \ref{subsec:spectrum} for a precise definition).

\vspace{2mm}

\paragraph*{\it Mapping class group action on character varieties}

A motivation for the introduction of our spectrum is that it seems to be a very
efficient tool for visualising the way the mapping class group acts on a representation $\rho$. We recall that the mapping class group $\mathrm{MCG}(\Sigma)$ is, for a compact surface $\Sigma$, the set of inner automorphisms of $\pi_1 \Sigma$. 

\vspace{2mm} The character variety is the space 
$$ \cv (\pi_1 \Sigma, \mathrm{SL}(2,\mathbb{R})) := \mathrm{Hom}(\pi_1 \Sigma, \mathrm{SL}(2,\mathbb{R})) \big/ \mathrm{SL}(2,\mathbb{R}) $$ of all representations of $\pi_1 \Sigma \longrightarrow \mathrm{SL}(2,\mathbb{R})$ up to conjugacy.

\vspace{2mm} A typical large element in $f \in \mathrm{MCG}(\Sigma)$ is of
\textit{pseudo-Anosov} type, which means that it is realised by a homeomorphism
$f$ of $\Sigma$ which dilates a given measured foliation $\mathcal{F}_+$ and
contracts another one which we call $\mathcal{F}_-$. A key feature of our
spectrum is that it has the following property. 

\begin{center}
\it
Let $\rho : \pi_1 \Sigma \longrightarrow \mathrm{SL}(2,\mathbb{R})$ be a
representation such that $\mathcal{F}_+ \notin \spe(\rho)$. Then the sequence $ [\rho \circ f^n] $ tends to infinity in $\cv(\pi_1 \Sigma, \mathrm{SL}(2,\mathbb{R}))$.
\end{center}

\noindent This statement can be made more precise, with an effective escape
rate, and also generalises to sequences $(f_n) \in \mathrm{MCG}(\Sigma)$ which
tend to infinity in controlled ways,
see Theorem~\ref{thm:repr_mcg}.

\vspace{2mm} \noindent Our results are restricted to the genus $1$ case, but they suggest that the following general conjecture ought to hold true.

\begin{conjecture}

Let $\Sigma_g$ be a genus $g \geq 2$ closed surface. For a generic $\rho : \Sigma_g
\longrightarrow \mathrm{SL}(2, \mathbb{R})$ which is not simple Fuchsian, the spectrum of $\rho$ is  closed and has empty interior, and locally has the structure of the product of a Cantor set by a codimension $1$ manifold.

\end{conjecture}

This conjecture, coupled with the fact above, suggests that given a representation and a sequence $f_n$ going to infinity in $\mathrm{MCG}(\Sigma_g)$, it is very likely that 
$$ [\rho \circ f_n ] \in \cv(\pi_1 \Sigma_g, \mathrm{SL}(2,\mathbb{R})) $$
diverges. This is because a large element $f_n \in \mathrm{MCG}(\Sigma_g)$ tends
to dilate a measured foliation $\mathcal{F}_n$, and since $\spe(\rho)$ is very
sparse, it is likely that $\mathcal{F}_n$ is in the complement of $\spe(\rho)$.

\vspace{2mm} \noindent On the other hand, it is a long-standing open problem to show that the action $\mathrm{MCG}(\Sigma_g)$ is \textit{ergodic} with respect to a measure in the Lebesgue class, in particular displays very strong recurrence properties. 
\noindent The discussion above suggests that if one is to find such recurrence for a given representation $\rho$, one better look at pseudo-Anosov mapping classes that dilate measured foliations in or near that in the spectrum of $\rho$.

\vspace{3mm}

\paragraph*{\it Cocycles over interval exchange maps}

We hope the discussion above will have convinced that if recurrence is to be found for an orbit of the action of $\mathrm{MCG}(\Sigma)$, one better look at mapping classes having to do with measured foliations in the spectrum of a given representation.

\vspace{2mm} What we think to be the most interesting feature of our approach is our interpretation of the study of the action of such mapping classes near the spectrum of a representation as a \textit{non-abelian extension} of the ergodic theory of \textit{interval exchange transformations} and its \textit{renormalization} theory. 

\vspace{2mm} Given a measured foliation $\mathcal{F}$, one can associate to it an \textit{interval exchange transformation} $T : [0,1] \longrightarrow [0,1]$ by means of taking a first return map on a transverse segment. Whether $\mathcal{F}$ belongs or not to the spectrum of $\rho$ can be reinterpreted in terms of the Lyapunov exponent of a piecewise constant cocycle $\rho : [0,1] \longrightarrow \mathrm{SL}(2,\mathbb{R})$ over $T$. Roughly speaking, such a piecewise constant cocycle is obtained from the image of $\rho$ on generators $\pi_1 \Sigma_g$ corresponding to the particular way $\mathcal{F}$ is obtained from $T$ as a suspension of $T$. We have the following important structural fact

\begin{center}
\it
$\mathcal{F}$ belongs to the spectrum of $\rho$ if (by definition) the Lyapunov exponent of $\rho$ over $T$ vanishes;
\end{center}

\vspace{2mm}
\paragraph*{\it Renormalisation} Rauzy and Veech introduced in the 80s a \textit{renormalization} procedure for IETs to study their ergodic properties. Their algorithm was extended to a cocycle over it to account for the study of piecewise constant cocycles over IETs taking their values in abelian groups, see \cite{Zorich}. This cocycle is often referred to as the Kontsevich-Zorich cocycle. In the particular case of $(\mathbb{R}, +)$-cocycles, one recovers what we want to call the "standard" ergodic theory of IETs, that is the study of Birkhoff sums over it. 

\vspace{2mm} Further structural work allows to interpret piecewise constant
cocycles as cohomology classes in some $\mathrm{H}^1(\Sigma, \mathbb{R})$ and
the action of the Kontsevich-Zorich cocycle as nothing but the action of some
mapping classes in $\mathrm{MCG}(\Sigma)$ on  $\mathrm{H}^1(\Sigma, \mathbb{R})$
via the symplectic representation.
This is a very quick summary of the technical
foundations of the field that is sometimes called ``Teichmüller dynamics''. 

\vspace{2mm} \noindent The key fact that we build upon is the following: 

\begin{center}
\it
The classes in $\mathrm{MCG}(\Sigma)$ which are candidates for recurrence in the orbit of a representation $\rho$ are precisely those given by the renormalization of the measured foliations in the spectrum of $\rho$.
\end{center}

What's more, when looking at the natural action of an extension of the Kontsevich-Zorich cocycle on piecewise constant $\mathrm{SL}(2, \mathbb{R})$ cocycle under renormalization, one recovers (more or less) the action of the mapping class group along the associated renormalization sequence. These observations put together give the following meta statement.

\vspace{3mm}
\fbox{
	\parbox{\textwidth}{ 
		\begin{center}
			The precise study of the mapping class group action on $\cv(\pi_1 \Sigma_g, \mathrm{SL}(2,\mathbb{R}))$ amounts to the renormalization theory of piecewise constant $\mathrm{SL}(2,\mathbb{R})$-cocycles over interval exchange maps.
		\end{center}
}}

\vspace{2mm}
\paragraph*{\it What is the gain?} At this stage one might be tempted to argue that we have swapped a difficult problem for another difficult one. This is rather true, and as such little can be gained without developing a working theory of $\mathrm{SL}(2,\mathbb{R})$-cocycles over IETs.  

\noindent However, the main gain seems to be that, even though hard to prove,
what \textit{ought to be true} seems much clearer in the world of cocycles.
Results on abelian cocycles, powered by the powerful methods from Hodge theory
yielding the hyperbolicity of the Kontsevich-Zorich cocycle, provide strong
hints at the existence of obstructions for a sequence $f_n \in
\mathrm{MCG}(\Sigma)$ corresponding to a measured foliation in the spectrum of
$\rho$ to display recurrence along the orbit of $\rho$. It should also suggest
what the codimension of such obstructions ought to be, and one could hope that recurrence is to be found in whatever remains. 

\vspace{2mm} \noindent That the results in the abelian case should reflect the picture in the non-abelian case is not pure fantasy. Indeed, the theory of $\mathrm{SL}(2, \mathbb{R})$-cocycles over rotations is now rather mature, and work of several authors (\cite{Herman,AvilaKrikorian,AvilaFayadKrikorian}) show that, although much harder to prove, the phenomena in the non-abelian case are the same as in the abelian case (in this particular case, an $\mathrm{SL}(2, \mathbb{R})$-cocycle over a rotation either has positive Lyapunov exponent or tends to be reducible to an $\mathrm{SO}(2)$-cocycle, which is a statement completely analogous to the solvability of the cohomological equation in the abelian case).

\vspace{2mm}
\paragraph*{\it Results}

Our results restrict to the more modest case of representations of the once-punctured torus, where we implement most of the general programme laid out above. They are all a consequence of our study of the extension of the Rauzy-Veech induction to $\mathrm{SL}(2,\mathbb{R})$-cocycles introduced in \ref{sec:renormalisation}. We denote by $\mathrm{T}^*$ the once-punctured torus.

\begin{theorem}
	Let $\rho : \pi_1 ( \mathrm{T}^*) \longrightarrow \mathrm{SL}(2,\mathrm{R})$
	be a
	non-simple Fuchsian representation which we further assume to be injective and that has no parabolic elements in its image. Then the spectrum of $\rho$ is a Cantor set.
\end{theorem}

The assumption that $\rho$ is injective and has no parabolic elements in its image is not too costly, the set of representations which do not satisfy this hypothesis is a countable union of strict submanifolds of the character variety. Furthermore, the result should still hold true without this hypothesis but it would significantly increase the combinatorial complexity (the reader who will have read through the proof in Section \ref{sec:cantor} should be easily convinced of this fact).

\vspace{2mm} Below are pictures representing the spectra of examples of
representations. The spectrum of a representation is formally defined as the set
of measured foliations for which the Lyapunov exponent of some cocycle vanishes: what is drawn below is the graph of such a Lyapunov exponent as a function of the measured foliation; the spectrum is the locus where it vanishes.

\begin{figure}[h]
\label{fig:graphs}
    \begin{subfigure}[t]{0.4\textwidth}
    \centering
    \includegraphics[width=1\textwidth]{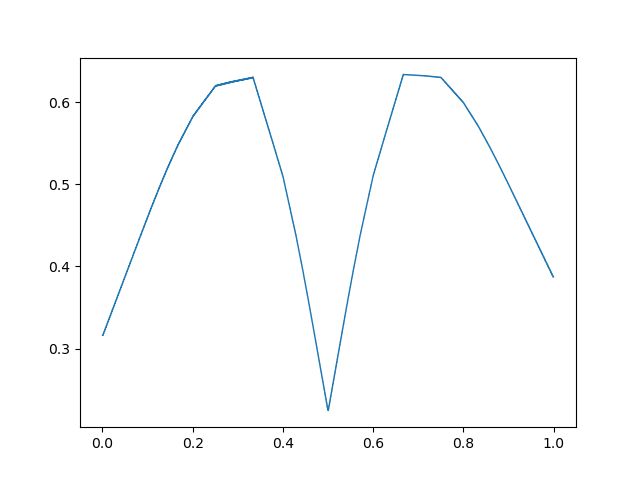} 
    \caption{Fuchsian representation}
    \end{subfigure}
    \begin{subfigure}[t]{0.4\textwidth}
    \centering
	\includegraphics[width=1\textwidth]{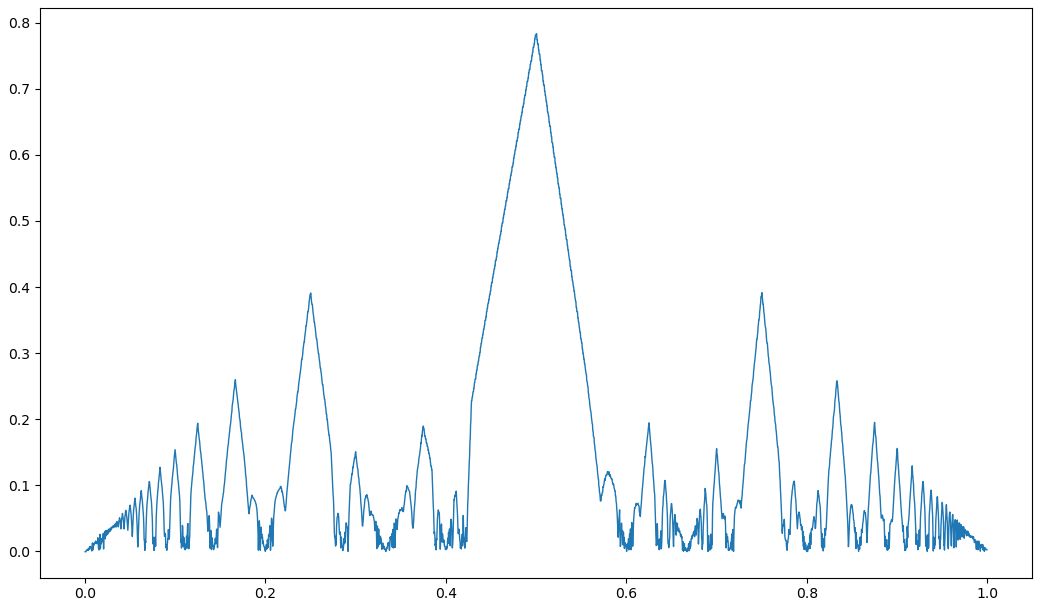}
    \caption{Generic representation}
    \end{subfigure}
     \begin{subfigure}[t]{0.4\textwidth}
    \centering
    \includegraphics[width=1\textwidth]{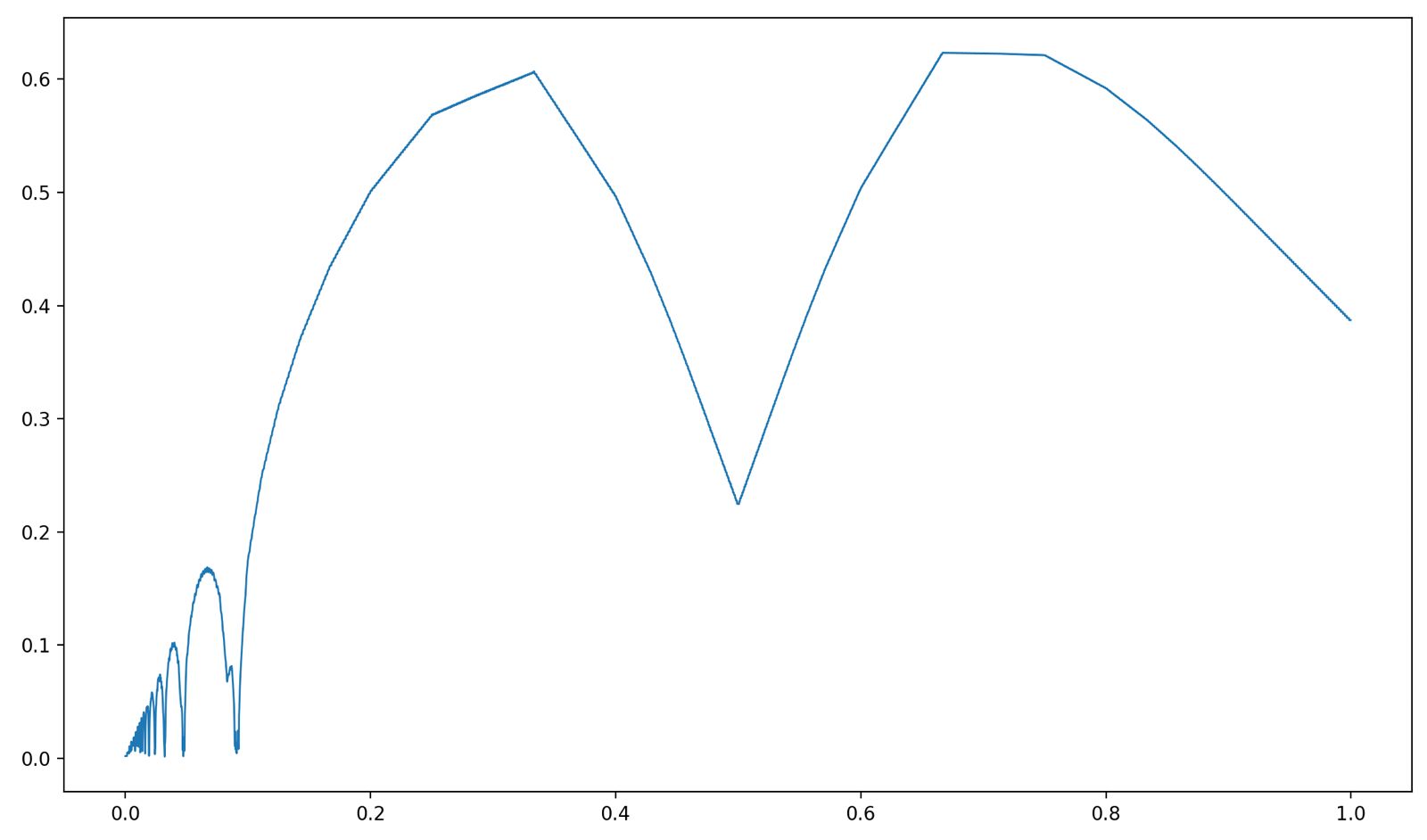}
    \caption{Yet another generic one}
    \end{subfigure}
\end{figure}

\vspace{2mm} \noindent Since the proof of this result goes through a careful analysis of the renormalization algorithm of the associated cocycle, we are able to derive the following description of the mapping class group action. 

\begin{theorem}
	Assume a measured foliation $\mathcal{F}$ belongs to the spectrum of $\rho$. Then if $(f_n)$ is a sequence of pseudo-Anosov maps in $\mathrm{SL}(2, \mathbb{Z}) = \mathrm{MCG}(\mathrm{T}^*)$ whose unstable foliations converge to $\mathcal{F}$, then 
	$$ [\rho \circ f_n] $$ eventually remains within a compact set of the relative character variety it belongs to.
\end{theorem}

This theorem can be reformulated in the language of the renormalization of
cocycles, and as such we think that it makes a convincing case for the idea that
$\mathrm{SL}(2,\mathbb{R})$-cocycles tend to behave like abelian cocycles. 

\begin{theorem}
	Let $\rho : [0,1] \longrightarrow \mathrm{SL}(2,\mathbb{R})$
	be a piecewise constant and generic cocycle, over $T$ a minimal $2$-IET.
	Then either
	\begin{itemize}
		\item the Lyapunov exponent of $\rho$ is positive and $\rho$ is \textbf{uniformly} hyperbolic;

		\item the Lyapunov exponent of $\rho$ vanishes and the iterated renormalizations of $\rho$ are uniformly bounded.
	\end{itemize} 
\end{theorem}

This theorem can be thought of as a non-abelian Denjoy-Koksma inequality. Indeed, renormalizations $\rho_n$ of the cocycle $\rho$ correspond to products of the form 

$$ \rho(T^{q_n-1}(x)) \cdots \rho(T^2(x))  \rho(T(x)) \rho(x) $$ and the theorem tells us that the logarithm of such a product either grows linearly in $q_n$ or is bounded (where the $q_n$ are the denominators of the good approximations of the rotation number of $T$).

\vspace{2mm} \noindent Finally, we further translate our results in the language of the dynamics of the Teichmüller flow lifted to the character variety bundle over the modular surface, as introduced in the article \cite{ForniGoldman}. We believe it to be a good language to formulate our results (it provides a continuous time renormalization operator for cocycles, in the same way the Teichmüller flow over the moduli space of translation surfaces provides a continuous analogue for the Rauzy-Veech induction). In the statement of the theorem $\mathcal{M}_c $ is the character variety bundle over the modular surface, and $(\Phi_t)$ is the lifted Teichmüller flow.

\begin{theorem}
The set of divergent orbits for the lift of the geodesic flow $ (\Phi_t) : \mathcal{M}_c \longrightarrow \mathcal{M}_c $ is an open and dense set. Its complement is exactly the set of recurrent orbits, and it is a lamination which is transversally a Cantor set. 
\end{theorem}

In particular it shows that the dynamics of $ (\Phi_t)$ on $ \mathcal{M}_c $ is not going to directly reflect that of the mapping class group, which is expected to act ergodically in our case  (unlike in the case of character varieties for compact Lie groups, see \cite{ForniGoldman}). We expect it to capture finer properties of this action, as our result shows.

\vspace{2mm}
\paragraph*{\it Related work, references and comments} A general reference for the case of $\mathrm{SL}(2,\mathbb{R})$-representations of the fundamental group of a once-punctured torus and the associated mapping class group action is the article of Goldman \cite{Goldman}.

\vspace{2mm} \noindent An important article to mention is~\cite{Bowditch}, which introduces a set which happens
to coincide with our spectrum
in the particular case of the relative character variety containing finite
volume hyperbolic structures on the punctured torus. The author mentions there
that this set is probably a Cantor set~\cite[p.729]{Bowditch}.
This set has been further studied in~\cite{tanEndInvariantsSL2C2008}
for representations with values in $ \SL(2,\C) $,
under the name of \emph{ends invariant} of a representation.
They also prove that this set is a Cantor set,
in the case of $ \SL(2,\R) $-representations that we study here.
The fact that our spectrum coincide with their ends invariant
(in the case of the once-punctured torus)
follows from
Theorem~\ref{thm:main_renorm}.
However, we believe that our definition of the spectrum
is more adapted to the more general setting of higher genus surfaces.

\vspace{2mm} \noindent The link between the action of the mapping class group
on the character variety of the punctured torus and the
Lyapunov exponents of some particular cocycles
has been exploited,
in the study of the spectral properties
of certain 1-dimensional Schrödinger operators
with quasi-periodic potential,
see~\cite{damanikSchrodingerOperatorsDynamically2017}
and references therein.

\vspace{2mm} \noindent Lyapunov exponents of representations
can be considered from the point of view of complex dynamics
on character variety, where they correspond to certain
\emph{Green functions}, see~\cite{girandDynamicalGreenFunctions2014a}.

\vspace{2mm} \noindent Let us remark that,
in contrast with the related works mentioned above,
the techniques that we use here do not rely on the explicit
trace coordinates of the character variety of the once-punctured torus,
nor on the particular structure of the complex of curves in this case,
but instead on elementary hyperbolic geometry considerations
and standard ideas from ergodic theory.
We hope that these methods will be more amenable to generalizations
to higher genus surfaces.
However, as such, this approach is limited to $ \SL(2,\R) $-representations.

\vspace{2mm} \noindent Our results imply that,
if $ \lambda $ is a lamination that does not belong
to the spectrum of a (generic) representation $ \rho $,
then
$ \rho $ is \emph{uniformly hyperbolic over $ \lambda $}
(more precisely, the lifted geodesic flow
introduced in Section~\ref{subsec:spectrum} is uniformly hyperbolic
restricted to the support of $ \lambda $).
This notion has been studied, in a vastly more general setting,
in~\cite{maloniDpleatedSurfacesTheir2023},
where it corresponds to the notion of
$ \lambda $-Anosov representations.
It is linked to the notion of pleated surface.


\paragraph*{\bf Organization of the article}
In Section~\ref{sec:def},
after recalling standard facts about the character variety,
we define the spectrum of a representation.
We explain how it can be studied using interval exchange transformations.

In Section~\ref{sec:coc_iet} and~\ref{sec:renormalisation},
we introduce the notion of a locally constant cocycle above an IET,
and explain a renormalization scheme adapted to this setting.

In Section~\ref{sec:cantor}, we prove our main results
for cocycles,
while in Section~\ref{sec:back_to_rep} we prove the
corresponding results for representations.

Section~\ref{sec:hyp_geom} and~\ref{sec:tech_lem}
contain elementary technical results that are used in
the proofs of the main results.

In Section~\ref{sec:flow},
we explain the consequences of our results
for a version of Teichmüller flow,
and we conclude in
Section~\ref{sec:open}
by a list of open questions.

\section{Spectrum of a representation}\label{sec:def}
 Everything we introduce in this paragraph could have been done in the general
 case of a representation of an arbitrary surface into $\mathrm{SL}(2,
 \mathbb{C})$. Since we end up only proving substantial results in the particular case of the once-punctured torus and $\mathrm{SL}(2, \mathbb{R})$, and that the description of the space of measured laminations simplifies drastically in this case, we decided to restrict ourselves to this case. 

\subsection{Character variety}

We denote by $\mathrm{T}^*$ the once-punctured torus, its fundamental group is
isomorphic to the free group on two generators
$ \left<a,b \right> $. We denote by 
$$ \cv(\mathrm{T}^*, \mathrm{SL}(2,\mathbb{R}))$$
the space of representations of the group $\pi_1 \mathrm{T}^*$ into $\mathrm{SL}(2,\mathbb{R})$ up to conjugacy. Formally it is defined as the quotient $\mathrm{Hom}(\pi_1 \mathrm{T}^*, \mathrm{SL}(2,\mathbb{R}))/\mathrm{SL}(2,\mathbb{R})$ where $\mathrm{SL}(2,\mathbb{R})$ acts by conjugacy.

\vspace{2mm}

\paragraph{\it Relative character varieties} Let $[\rho]$ be a class in $
\cv(\mathrm{T}^*, \mathrm{SL}(2,\mathbb{R}))$ represented by some representation
$\rho$.  Consider the free homotopy class of loops representing the (oriented)
boundary component of $\mathrm{T}^*$.
It corresponds to the conjugacy class in $\pi_1 \mathrm{T}^*$
of the commutator
$ \left[ a,b \right] $,
which we denote by $\mathbf{c}$.
The trace of $\rho(\left[ a,b \right])$ only depends on $[\rho]$. We define 
\[ \cv_c(\mathrm{T}^*, \mathrm{SL}(2,\mathbb{R}))\]
the \textit{relative character variety} which consists of those classes of
representations which map a boundary curve onto an element of trace $c \in
\mathbb{R}$,
that is such that
$ \tr \rho(\left[ a,b \right]) = c $.
The behavior of the representation
depends strongly on this trace
$ c $,
and we describe the different possible cases
(see~\cite{Goldman}).

The first case is $ c < 2 $:
the relative character variety
$\cv_c(\mathrm{T}^*, \mathrm{SL}(2,\mathbb{R}))$
has four connected components,
each of which is isomorphic to
a Teichmüller space of hyperbolic structures
on the once-punctured torus,
with different behaviors at the puncture depending
on $ c $.
The action of the mapping class group on these
Teichmüller components is properly discontinous.

When
$ c < -2 $,
they correspond to hyperbolic structures
with a funnel,
and when
$ c = -2 $
to hyperbolic structures with a cusp.
In these cases,
the representations
are faithful and discrete.
When
$ c \in (-2,2) $,
they correspond to hyperbolic structures
with a conical singularity at the puncture.
It is remarkable that
in that case,
the representations are
not discrete,
even though the
action of the
mapping class group is properly discontinous.
This is explained by the fact
that they are \emph{simple Fuchsian},
in the terminology of~\cite{TholozanWang}.

The case $ c = 2 $ is slightly degenerate,
and correspond to representations that take
values in an abelian group,
and we won't consider it in this text.
There are still four components,
and up to taking a ramified covering of degree 2,
the action of the mapping class group on
each of these components
is isomorphic to the linear action of
$ \SL(2,\Z) $ on $ \R^{2} $,
which is well understood.

The remaining case
$ c > 2 $
is the one we are mainly interested in.
In that case the relative character variety is connected.
There exists an open set $ U $ (empty if $ c < 18 $)
consisting of representations
that are holonomies of hyperbolic structures on a \emph{pair of pants}:
indeed, a pair of pants has the same fundamental group as the
once-punctured torus.
The mapping class group acts properly discontinously on this open set.
It is only on the complement of $ U $
that interesting dynamics can take place,
and Goldman showed that the action is \emph{ergodic}
there
(with respect to the Lebesgue measure).
Consequently, we only consider the case
$ c > 0 $ in the remaining of the text.
This has the following geometric consequence:
if $ A = \rho(a) $ and $ B = \rho(b) $
are hyperbolic isometries,
then the axis of $ A $ and $ B $
do not intersect
(as this classically implies that
$ \tr \left[ A,B \right] < 2 $).

In our specific context,
we define the
\emph{simple Fuchsian} representations
to be
the case $ c < 2 $,
and the representations that belong to $ U $
for $ c > 2 $
(see~\cite{tholozanSimpleAnosovRepresentations2023}
for the general definition).

\vspace{2mm} It is probably worth mentioning that for $c \neq 2$, $ \cv_c(\mathrm{T}^*, \mathrm{SL}(2,\mathbb{R}))$ together with its quotient topology is a (non-compact) topological surface.
It is homeomorphic to a union of connected components of the affine surface in
$ \R^3 $ defined by the equation
\[
	x^2 + y^2 + z^2 - xyz = c + 2,
\]
the so-called \emph{Markov surface},
see~\cite{Goldman}.

\subsection{Space of measured laminations and measured foliations}

Now consider the once-punctured torus $\mathrm{T}^*$, which we will think of as the compact torus $\mathrm{T}$ with one marked point $p$, so that $\mathrm{T}^* = \mathrm{T} \setminus \{p\}$. 

\vspace{2mm}

\paragraph{\it Measured foliations} Endow $\mathrm{T}$ with an identification with $S^1 \times S^1 = \mathbb{R}^2 / \mathbb{Z}^2$. A \textit{measured foliation} on $\mathrm{T}^*$ is just a linear foliation on $\mathbb{R}^2 / \mathbb{Z}^2$. A linear foliation in this context is the trace on $\mathbb{R}^2 / \mathbb{Z}^2$ of the foliation of $\mathbb{R}^2$ by straight lines of a given slope. 

\vspace{2mm} As such, the space of (oriented) measured foliations naturally
identifies with $S^1$, the unit circle.

\vspace{2mm}

\paragraph{\it Measured laminations} We recall the following fact. Fix a
finite-volume complete hyperbolic metric on $\mathrm{T}^*$. Measured foliations
are in one-to-one correspondence with \textit{ergodic} measured geodesic
laminations on $\mathrm{T}^*$. That is not very important for our purpose, but
since some of the objects we are working with are traditionally more often
introduced in terms of geodesic laminations, we mention it for the reader who is more used to this language.

\subsection{Spectrum} 
\label{subsec:spectrum}

We now define in an abstract way the \textit{spectrum} of a given representation. We will see later that it can be computed in a rather explicit way, using products of matrices used to define the representation. 

\vspace{2mm} Given a representation $\rho : \pi_1 \mathrm{T}^* \longrightarrow \mathrm{SL}(2,\mathbb{R})$, we can define the flat vector bundle 

$$ E_{\rho} := (\widetilde{\mathrm{T}^*} \times \mathbb{R}^2)/ (x, \vec{v}) \sim
(\gamma \cdot x, \rho(\gamma) \cdot \vec{v}).$$ It defines a smooth rank $2$
vector bundle over $\mathrm{T}^*$ endowed with a flat connection. The smooth
bundle itself extends to $\mathrm{T}$ (but the flat connection doesn't), we endow this extension with a smooth metric $g$. 

\vspace{2mm} \paragraph{\it The lifted dynamics} Endow $\mathrm{T}$ with a flat
metric. The geodesic flow restricted to $\mathrm{T}^*$ naturally lifts to a dynamical system to $E_{\rho}$, thanks to the parallel transport given by the flat connection on $E_{\rho}$.

\vspace{2mm} \paragraph{\it Lyapunov exponent}Now consider a measured foliation
$\mathcal{F}$ on $\mathrm{T}$. There exists an angle $\theta$ such that
$\mathcal{F}$ is the foliation by straight line of the flat torus $\mathrm{T}$.
The geodesic flow therefore gives a natural flow on $\mathrm{T}$ whose leaves
are exactly that of $\mathcal{F}$, we denote this flow by $(\varphi_t )_{t\in
\mathbb{R}}$. Given $x \in \mathrm{T}^*$, we denote by 
\[ \phi_t(x) : E_{\rho}(x) \longrightarrow E_{\rho}(\varphi_t(x)) \]
the map induced from the fibre at $x$ to that at $\varphi_t(x)$. Its a linear map and we denote by $||\phi_t(x)||$ the norm defined using the norms induced our arbitrary choice of a metric $E_{\rho}$.

\begin{definition}[Lyapunov exponent]

The Lyapunov of the representation $\rho$ with respect to the measured lamination $\mathcal{F}$ is 
$$ \chi(\rho, \mathcal{F}) := \lim_{t \rightarrow + \infty}{\frac{1}{t}\int_{x \in \mathrm{T}}{\log ||\phi_t(x)||dx}}.$$

\end{definition}

\vspace{2mm} \paragraph{\it Spectrum} One will have noticed that as such, our definition of the Lyapunov exponent depends on a choice of flat metric on $\mathrm{T}$ (which then serves to parametrise the measured foliation $\mathcal{F}$. However, the choice of a volume $1$ metric makes the exponent well-defined, and we will make this assumption. Another remark is that there are many other ways of defining the Lyapunov exponent, using a hyperbolic metric on $\mathrm{T}^*$ and measured laminations instead. At the end of the day, all one needs is a way of parametrising laminations/foliations that, when coupled to the transverse measure, gives total mass one to the total space and we have just chosen one.

\begin{definition}[Spectrum]

Let $\rho : \pi_1 \mathrm{T}^* \longrightarrow \mathrm{SL}(2, \mathbb{R})$. The spectrum of $\rho$ is the set 

$$ \spe(\rho) := \{ \theta \in S^1 \ | \ \chi(\rho, \mathcal{F}_{\theta}) = 0 \}.$$ 

\end{definition}

\subsection{Mapping class group action and the spectrum}

In this paragraph, we explain how the spectrum changes under the action of the
mapping class group. We recall that $\mathrm{T} = \mathbb{R}^2 / \mathbb{Z}^2$,
we let $p$ be our marked point be $p = (0,0)$ so that $\mathrm{T}^*$ is effectively $(\mathbb{R}^2 / \mathbb{Z}^2) \setminus \{(0,0)\}$. The mapping class group of $\mathrm{T}^*$ can be realised by the action of $\mathrm{SL}(2, \mathbb{Z})$ on $(\mathbb{R}^2 / \mathbb{Z}^2)$ (by that we mean that $\mathrm{SL}(2, \mathbb{Z})$ is a subgroup of the group of orientation-preserving homeomorphisms fixing $p$, which projects bijectively onto the mapping class group of $\mathrm{T}^*$). Furthermore, it acts on the space of oriented measured laminations as the projective action of $\mathrm{SL}(2, \mathbb{Z})$ on the set of half lines of $\mathbb{R}^2$.

\vspace{2mm}  Let $\theta \in S^1$ and $A \in \mathrm{SL}(2,\mathbb{Z})$.  We have the following.

\begin{proposition}

For all $\theta \in S^1$ and $A \in \mathrm{SL}(2,\mathbb{Z})$ we have 
$$ \chi(\rho, \mathcal{F}_{\theta}) = 0 \Leftrightarrow \chi(\rho \circ A^{-1}, \mathcal{F}_{A \cdot \theta}) = 0. $$

\end{proposition}

\begin{proof}

It suffices to note that $A$ maps $\mathcal{F}_{\theta}$ onto $\mathcal{F}_{A \cdot \theta}$ and the flat fibre bundle $E_{\rho}$ onto $E_{\rho \circ A^{-1}}$. Since positivity of the Lyapunov exponent for a $1$-dimensional foliation does not depend on the parametrisation (a change of parametrisation multiplies the Lyapunov exponent by a positive constant), the proposition is proved.

\end{proof}

\subsection{Reduction to interval exchange maps}

In order to study the Lyapunov exponent more directly, we are going to reduce its study to that of special cocycles over rotations, which we will in turn interpret in terms of interval exchange maps (see paragraph \ref{subsec:iets} for precise definitions).

\vspace{2mm} \paragraph{\it First-return map} Let $\mathcal{F}$ be an oriented
arbitrary measured foliation with transverse measure $\mu$. There always exists
an embedded interval circle $\gamma$ in $\mathrm{T}^*$ that is transverse to $\mathcal{F}$ such that the first return map of $\mathcal{F}$ is a rotation. Precisely, $\gamma$ inherits a measure from $\mathcal{F}$ which gives a natural identification between $\gamma$ and $\mathbb{R}/ \mu(\gamma
) \mathbb{Z}$.

\vspace{2mm} Informally, we now think of this first-return map as an \textit{interval exchange transformation} (see paragraph \ref{subsec:iets}), by thinking of the circle $\mathbb{R}/ \mu(\gamma
) \mathbb{Z}$ as whose $[0, \mu(\gamma)]$ extremities have been glued together.
There are many ways of doing this (a choice of a point on $\gamma$ essentially) and we choose the point $q$ corresponding to the first intersection of the backward orbit of $\mathcal{F}$ through $p$.
In this coordinate, the first return map of $\mathcal{F}$ is an interval exchange transformation, which is an invertible map from $[0,\mu(\gamma)]$ to $[0,\mu(\gamma)]$ with one discontinuity which is a translation restricted to its continuity intervals. 

\begin{figure}[!h]
	\begin{center}
		\def\svgwidth{0.4 \columnwidth}
\begingroup%
  \makeatletter%
  \providecommand\color[2][]{%
    \errmessage{(Inkscape) Color is used for the text in Inkscape, but the package 'color.sty' is not loaded}%
    \renewcommand\color[2][]{}%
  }%
  \providecommand\transparent[1]{%
    \errmessage{(Inkscape) Transparency is used (non-zero) for the text in Inkscape, but the package 'transparent.sty' is not loaded}%
    \renewcommand\transparent[1]{}%
  }%
  \providecommand\rotatebox[2]{#2}%
  \newcommand*\fsize{\dimexpr\f@size pt\relax}%
  \newcommand*\lineheight[1]{\fontsize{\fsize}{#1\fsize}\selectfont}%
  \ifx\svgwidth\undefined%
    \setlength{\unitlength}{393.85645015bp}%
    \ifx\svgscale\undefined%
      \relax%
    \else%
      \setlength{\unitlength}{\unitlength * \real{\svgscale}}%
    \fi%
  \else%
    \setlength{\unitlength}{\svgwidth}%
  \fi%
  \global\let\svgwidth\undefined%
  \global\let\svgscale\undefined%
  \makeatother%
  \begin{picture}(1,1.21916718)%
    \lineheight{1}%
    \setlength\tabcolsep{0pt}%
    \put(0,0){\includegraphics[width=\unitlength,page=1]{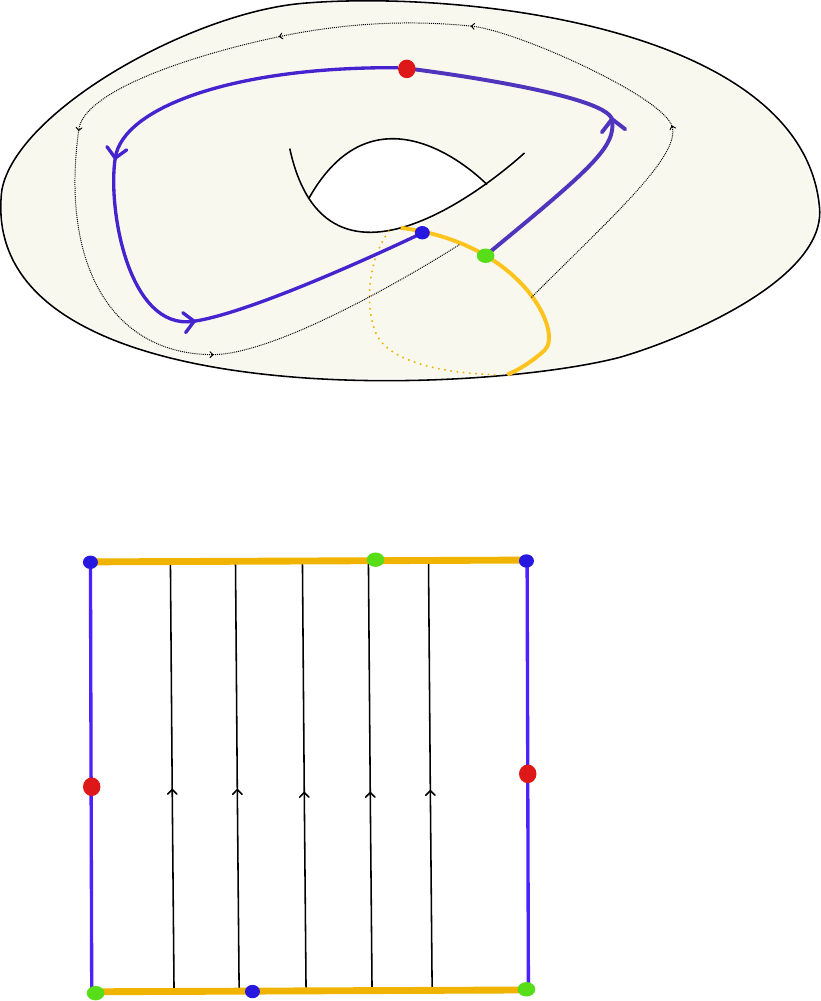}}%
    \put(0.66493118,0.80713626){\makebox(0,0)[lt]{\lineheight{1.25}\smash{\begin{tabular}[t]{l}$\gamma$\end{tabular}}}}%
    \put(0.50649742,1.15881141){\makebox(0,0)[lt]{\lineheight{1.25}\smash{\begin{tabular}[t]{l}$p$\end{tabular}}}}%
    \put(0.59171554,0.89355477){\makebox(0,0)[lt]{\lineheight{1.25}\smash{\begin{tabular}[t]{l}$q$\end{tabular}}}}%
  \end{picture}%
\endgroup%

	\end{center}
\caption{Moving from foliation to a map of the interval.}
\label{fig:nearp2}
\end{figure}

\vspace{2mm} \paragraph{\it Cocycle over $I$} Let $I$ be the interval $\gamma
\setminus \{q \}$. Because $I$ is contractible, the restriction of $E_{\rho}$ to
$I$ canonically identifies with $I \times \mathbb{R}^2$ thanks to the flat
connection. The first-return map of the lifted dynamics on $I \times \mathbb{R}^2$ has to be of the form 

$$ (x,v) \mapsto (T(x), M\cdot v) $$ on continuity intervals (as the flat
connection forces the second coordinate to be constant on continuity intervals),
where $M \in \mathrm{SL}(2, \mathbb{R})$. To summarise, if $\mathcal{F}$ is a
given measured foliation and $I$ an interval as above we have obtained:
\begin{itemize}
\item an interval exchange map on two intervals $T_0 : [0,\mu(I)] \longrightarrow [0,\mu(I)]$;
\item a pair of matrices $(A,B)$, each of which being associated to one of the two intervals of $T_0$.
\end{itemize}

\vspace{2mm} \paragraph{\it Family of cocycles} Given our circle $\gamma$, we
have defined for all measured foliations $\mathcal{F}$ a map $T$ which is a $2$-interval exchange transformation together with a locally constant cocycle parametrised by two matrices $(A,B)$. Recall that the set of measured foliations identifies with the circle $S^1$, and that the set of measured foliations transverse to $\gamma$ is an open set $V(\gamma) \subset S^1$, so that we have defined a map 

$$ \begin{array}{ccc}
V(\gamma) & \longrightarrow & \{2-\text{interval exchange transformations} \} \\
\theta & \longmapsto & T_{\theta} 
\end{array}$$ where $T_{\theta}$ is the $2$-IET corresponding to the first
return map of the foliation by straight lines of slopes directed by $\theta \in
\mathbb{S}^1$. This map is nothing but the map that associates the rotation that
is the first-return map of $\mathcal{F}_{\theta}$ onto $\gamma$. For instance, if $\gamma$ is a horizontal closed geodesic in $\mathrm{T} = \mathbb{R}^2/ \mathbb{Z}^2$, this map is nothing but $\theta \longmapsto (x \mapsto x + \tan \theta)$. In particular, this identification is an increasing map so it's locally injective.

\vspace{2mm} 

\vspace{2mm} \fbox{
\parbox{\textwidth}{ 
The upshot of this formal construction is that we have reduced the study of the cocycle defined by a representation $\rho : \mathrm{T}^* \longrightarrow \mathrm{SL}(2, \mathbb{R})$ over the space of all measured foliations on $\mathrm{T}$, to a family of piecewise constant cocycle over the space of $2$-interval exchange transformations. The one thing one has to be careful with is that one needs a finite number of such families of $2$-IETs to cover all measured foliations.
}}

\vspace{2mm}

\section{Cocycles over interval exchange transformations}\label{sec:coc_iet}

\subsection{Rotations as interval exchange maps}
\label{subsec:iets}
When narrowing down our study to the case of once-punctured tori, measured laminations become equivalent to linear foliations on the closed torus, the study of which reducing to that of rotations on the circle (by means of taking a first return map). 
\noindent The rotation of angle $\alpha \in S^1 = \mathbb{R}/ \mathbb{Z}$ is by definition the map
\[
	\begin{array}{ccccc}
		R_{\alpha} & := & S^1 & \longrightarrow & S^1 \\
			   && x & \longmapsto & x + \alpha
	\end{array}
\]

For technical reasons (which have a conceptual explanation which we save for later) it will be more convenient to think of $R_{\alpha}$ as the map that it induces on $[0,1]$;
by abuse of notation, we also denote
\[
	\begin{array}{cccccc}
		R_{\alpha} & := & [0,1] & \longrightarrow & [0,1] & \\
			   && x & \longmapsto & x + \alpha & \text{if} \ x \in [0,1- \alpha] \\
			   && x & \longmapsto & x + \alpha - 1 & \text{if} \ x \in (1- \alpha,1]
	\end{array}.
\]

\noindent We denote the continuity intervals of $R_{\alpha}$ by $I_a := [0, 1- \alpha]$ and $I_b := [1- \alpha, 1]$ respectively. We have thus realised $R_{\alpha}$ as an \textit{interval exchange transformation} (IET) on $2$-intervals. 
\vspace{2mm}

\subsection{Cocycles over an IET}

In this paragraph we consider $T : [0,1] \longrightarrow [0,1]$ an IET on two intervals. Consider a map $\rho : [0,1] \longrightarrow \mathrm{SL}(2, \mathbb{R})$ (note that we are using here the same notation for such a map as we are using for a representation $\rho : \pi_1 \mathbb{T}^* \longrightarrow \mathrm{SL}(2, \mathbb{R})$ which is a deliberate choice).

\vspace{2mm}

In the literature, one would often call such a map $\rho$ a \textit{cocycle}. The reason for that is that we are going to be interested in studying products of the form 
$$ \rho_n(x) =  \rho(T^{n-1}(x)) \cdot \rho(T^{n-2}(x)) \cdots \rho(T^{2}(x)) \cdot \rho(T(x)) \cdot \rho(x).$$
\noindent A more formal way of dealing with such products would be to introduce a fibred dynamical system 
$$ (x,v) \longmapsto (T(x), \rho(x) v)$$
and note that, under iteration, the second variable records the action of the product $\rho_n(x)$ on $\mathbb{R}^2$. Since there shall be no practical gain for us to follow this formal route, we will do away with it altogether. 

\begin{definition}
A cocycle (over $T$) is a map $\rho : [0,1] \longrightarrow \mathrm{SL}(2,\mathbb{R})$
\end{definition}

The only reason why we bother calling such a map a cocycle is because we want our reader's mind focused on the fact that we will be considering products of the form $ \rho_n(x) =  \rho(T^{n-1}(x)) \cdot \rho(T^{n-2}(x)) \cdots \rho(T^{2}(x)) \cdot \rho(T(x)) \cdot \rho(x).$ 

\vspace{2mm} Now in this article, we make the following standing assumption. When $T$ is an interval exchange transformation,

\begin{center}
\bf
All cocycles $\rho$ considered will be assumed to be constant on the continuity intervals of $T$. 
\end{center}

This way, a cocycle $\rho$ over an IET on $2$ intervals is completely
characterised by the value it takes over $I_a$ and over $I_b$. In other words,
such a cocycle $\rho$ is the same thing as the datum of a pair of matrices $(A,B) \in \mathrm{SL}(2, \mathbb{R})^2$.

\subsection{Lyapunov exponent}

Cocycles can be thought of as a generalization of classical ergodic theory; the cocycle $\rho : [0,1] \longrightarrow \mathrm{SL}(2, \mathrm{R})$ playing the role of an observable only taking its values in the group $\mathrm{SL}(2, \mathrm{R})$ instead of $\mathbb{R}$. Products

$$ \rho_n(x) =  \rho(T^{n-1}(x)) \cdot \rho(T^{n-2}(x)) \cdots \rho(T^{2}(x)) \cdot \rho(T(x)) \cdot \rho(x)$$ are thus seen to be natural generalization of Birkhoff sums. 

\vspace{2mm} \noindent As in classical ergodic theory, the first interesting numerical invariant to look at is the average growth of such  products. For $x \in [0,1]$, one introduces 

$$ \chi(x, \rho) := \limsup \frac{\log ||  \rho_n(x) || }{n}.$$ We call it the \textit{Lyapunov exponent} of $\rho$ at $x$.

\noindent We remind our reader of the Oseledets theorem which in our case is going to allow us to rid ourselves of the dependency in $x$.

\begin{theorem}[Oseledets]
Let $T : [0,1] \longrightarrow [0,1]$ be an ergodic $2$-IET (condition which is
equivalent to the associated rotation having irrational angle). Then there exists a number $\chi(\rho)$ such that for almost every $x \in [0,1]$, 
$$\chi(x, \rho) = \chi(\rho).$$
\end{theorem}

\subsection{Back to measured foliations}

Recall that to a representation $\rho$, a measured foliation $\mathcal{F}$ and a well-chosen transverse interval $I$, we had associated 

\begin{enumerate}
\item $T$ a $2$-IET obtained as a first-return map of $\mathcal{F}$ on $I$;
\item a particular pair of matrices $(A,B)$ that generated the image of $\rho$.
\end{enumerate}

By construction, we have:
\begin{proposition}
The Lyapunov exponent of $\rho$ with respect to the measured foliation $\mathcal{F}$ is equal to the Lyapunov exponent of the piecewise constant cocycle associated to the pair $(A,B)$ over $T$.

\end{proposition}



\section{Renormalisation of cocycles}\label{sec:renormalisation}

\subsection{Rauzy induction} 

In this paragraph we introduce an algorithm which takes as input an $T_0$ IET and outputs a sequence of IETs $T_0, T_1, \cdots, T_n, \cdots $. The two main features of this sequence are the following:

\begin{itemize}
\item the $T_n$s have the same number of continuity intervals as $T_0$;
\item $T_n$ is a (normalised) first-return map of $T_0$ onto an interval of the form $[0,x_n]$.
\end{itemize}

We restrict ourselves to the case of $2$-IETs, for an introduction to the theory for arbitrary IETs see \cite{Yoccoz}.

\vspace{2mm}

\paragraph*{\bf Elementary step of the induction.} Let $R_{\alpha} = T_0 : [0,1] \longrightarrow [0,1]$ be a $2$-IET with respective continuity intervals $I_a = [0,1 - \alpha]$ and $I_b = [1- \alpha, 1]$. 

\begin{definition} 
For such a $2$-IET $T_0$, we define a new $2$-IET $\mathcal{R}(T_0)$ the following way

\begin{itemize}
\item If $|I_a| < |I_b|$, define 

$$ \mathcal{R}(T_0) := \ \text{first return of} \ T_0  \ \text{on the interval} \ [0,\alpha].$$

\item If $|I_a| > |I_b|$, define 

$$ \mathcal{R}(T_0) := \ \text{first return of} \ T_0  \ \text{on the interval} \ [0,1-\alpha].$$ 
\end{itemize}

\end{definition}

In the first case, we say that \textit{top wins} (and \textit{bottom loses})   whereas in the second case, we say that \textit{bottom wins} (and that \textit{top loses}). In each case, the largest interval is called the \textit{winner} and the smallest is called the \textit{loser}. This terminology is inspired by the following picture.

\begin{figure}[!h]
	\begin{center}
		\def\svgwidth{0.4 \columnwidth}
			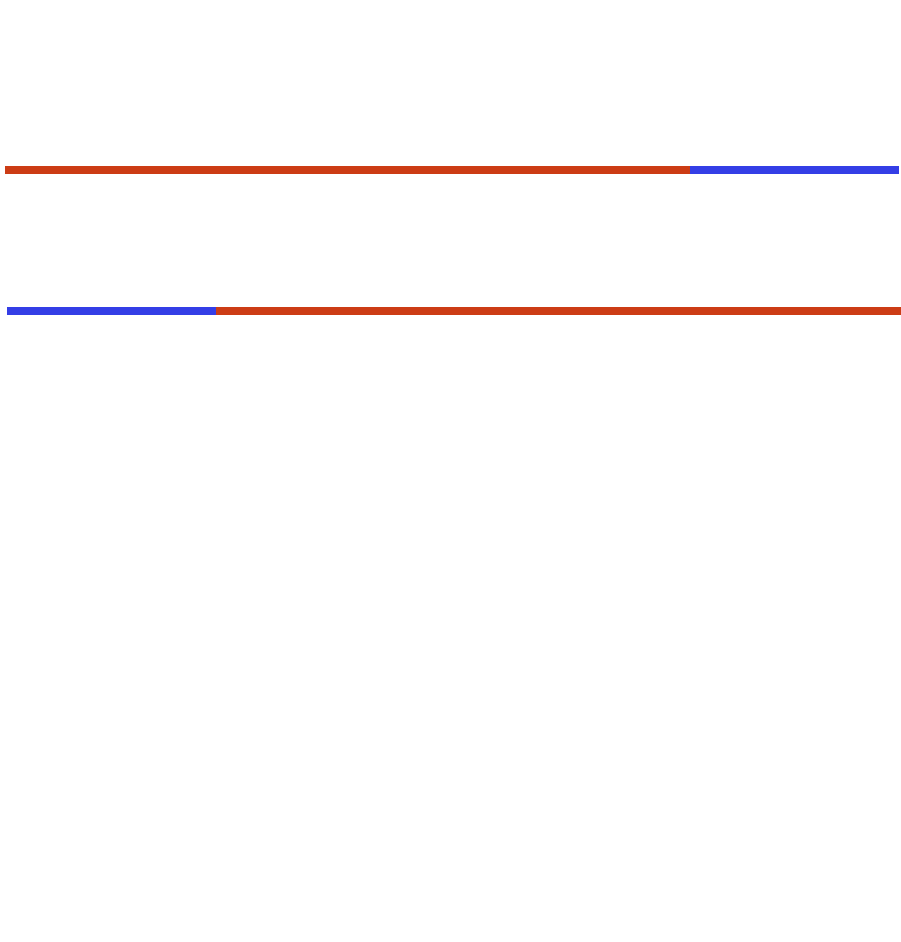
	\end{center}
\caption{Examples of configurations where bottom or top wins}
\label{fig:nearp}
\end{figure}

\vspace{2mm}

\paragraph*{\bf The algorithm per se} Once we are equipped with the elementary step of the induction, given  $T_0$ a $2$-IET (which again is the same thing as a rotation $R_{\alpha}$ for some $\alpha \in S^1$), we define two sequences. 

\vspace{2mm}

\begin{itemize}
\item a sequence $(T_n)_{n \in \mathbb{N}}$ where $T_n$ is defined recursively by $T_{n+1} = \mathcal{R}(T_n)$ (where $\mathcal{R}(T_n)$ has been normalised so that it is a map from $[0,1]$ to itself). We call the sequence $(T_n)$ the \textit{renormalizations} of $T_0$ and the $T_n$ the \textit{$n^{th}$ renormalization} of $T_0$.
\item a sequence $\gamma =(\gamma_n)_{n \in \mathbb{N}}$ of letters in the alphabet $\{t,b\}$, where $\gamma_n$ is equal to \textit{t} or \textit{b} depending on whether top or bottom wins for $T_n$.
\end{itemize}

We are going to see in the next paragraph that the datum of the sequence $\gamma$ is more or less equivalent to the continued fraction expansion of $\alpha$.

\subsection{Renormalisation of cocycles}

In this paragraph we enhance the Rauzy induction to keep track of the evolution of a cocycle $\rho$ under iterations of $T_0$. Since an element $T_n$ of the sequence of renormalizations of $T_0$ is a first-return map of $T_0$, we are going to want to know what has happened to the cocycle $\rho$ under the iterations of $T_0$ defining $T_n$. By that we mean keeping track of the product $\rho(T^{k_n-1}(x)) \cdots \rho(T^{2}(x)) \rho(T(x)) \rho(x)$ where 

\begin{itemize}

\item $x$ is in the interval $[0,x_n]$ such that $T_n$ is the first return on $[0,x_n]$;

\item $k_n$ is the integer such that $T_n(x) = T_0^{k_n}(x)$.

\end{itemize}

Let $T_0$ be a $2$-IET, $(T_n)$ its sequence of renormalization and $\gamma(T_0)$ its Rauzy path. Consider a cocycle $\rho_0 : [0,1] \longrightarrow \mathrm{SL}(2,\mathbb{R})$ which we recall we have assumed to be constant on continuity intervals of $T_0$, so that $(\rho_0)_{|I_a} \equiv  A_0$ and $(\rho_0)_{|I_b} \equiv B_0$ for a pair of matrices $(A_0,B_0) \in \mathrm{SL}(2,\mathbb{R})^2$.

\begin{definition}
Let $T_0$, $\rho_0 = (A_0,B_0)$, $(T_n)$ and $\gamma(T_0)$ be as above. We define the sequence of renormalized cocycles $(\rho_n = (A_n,B_n))_{n \in \mathbb{N}}$ inductively the following way

\begin{itemize}
	\item if $\gamma_n = b$, $(A_{n+1}, B_{n+1}) = \tau_1 (A_n, B_n) := (A_n, B_n \cdot A_n)$.
	\item if $\gamma_n = t$, $(A_{n+1}, B_{n+1}) = \tau_2 (A_n, B_n) := (B_n \cdot A_n, B_n)$.
\end{itemize}

\end{definition}

\subsection{Acceleration and continued fraction expansions}

We now introduce an acceleration of the induction, which we will use later on in paragraph \ref{subsec:zeroLyap}. This material is classical and so we don't give any proofs for most of the statements, which doesn't mean that they are completely trivial (although none are deep theorems and can be derived from the definitions).

\vspace{2mm} Consider a $2$-interval exchange transformation $T$. We define 

$$ \mathcal{Z}(T) := \mathcal{R}^{b(T)}(T) $$ where $b(T)$ is the smallest integer such that $\gamma_{b(T)}$ is not equal to $\gamma_1$. In other words, we group together all steps of renormalization where the same interval wins and add one for good measure. We call $\mathcal{Z}(T)$ the \textit{accelerated} renormalization of $T$. 
\vspace{2mm}

\paragraph*{\it Continued fractions} We briefly recall here the concept of the continued fraction expansion of a real number. Recall that for any irrational number $\alpha \in [0,1)$ one can associate a sequence $a_1, a_2, \cdots, a_n$ of positive integers such that 

$$ \alpha = \frac{1}{a_1 + \frac{1}{a_2 + \frac{1}{a_3 + \cdots}}}.$$ We will use the traditional notation $(q_n)$ for the sequence defined inductively by $q_0 = 1, q_1 = a_1$ and 

$$ q_{n} = a_n q_{n-1} + q_{n-2}.$$ We will be using the following fact later. 

\begin{lemma}
For all $\alpha$ irrational and for all $n$, $q_n \geq \sqrt{2}^n$.
\end{lemma}

\begin{proof}

This is a simple consequence of the formula $$ q_{n} = a_n q_{n-1} + q_{n-2}$$ which implies that $q_n \geq 2 q_{n-2}$ for all $n \in \mathbb{N}$.

\end{proof}

\paragraph*{\it Link with the acceleration} We explain the link between continued fractions and renormalization. Let $T_{\alpha}$ be the $2$-IET representing the rotation of angle $\alpha$. We are going to assume that for $T_{\alpha}$ bottom wins. In this case, we have 

$$ a_1 = b(T_{\alpha}) $$ \textit{i.e.} the first coefficient in the continued
fraction expansion of $\alpha$ is the number of times the standard induction has to be accelerated to get $\mathcal{Z}(T)$. We then have 

$$ a_n = b(\mathcal{Z}^n(T_{\alpha})).$$ 

\paragraph*{\it First-return time} For all $n \in \mathbb{N}$, we have that $\mathcal{Z}^n(T_{\alpha})$ is the first-return of $T_{\alpha}$ on an interval $I_{n} = [0,\beta_n]$. $I_n$ itself splits into two intervals of continuity for $\mathcal{Z}^n(T_{\alpha})$. We have the following.

\begin{proposition}
\label{prop:returntimes}
For all $n \in \mathbb{N}$, $\mathcal{Z}^n(T_{\alpha})$ is the restriction of $T_{\alpha}^{q_n}$ to one of its intervals of continuity and of $T_{\alpha}^{q_n + q_{n-1}}$ to the other.

\end{proposition}

%
%
%
%
%
%
%
%
%
%
%
%
%
%

\subsection{Link with the action of the Mapping Class Group}\label{sub:traj_mcg}

We now explain that the sequence of cocycles
$ (A_n, B_n) $
generated from $ (A,B) $
by the renormalization of the IET
corresponds to applying a certain sequence
of element of the mapping class group
to the initial representation $ \rho $.

This follows simply from the remark that
if $ A = \rho(a) $ and $ B = \rho(b) $
are the image of the standard generators of the torus,
then the transformation
$ (A,B) \mapsto (A, BA) $
corresponds
to applying a Dehn twist
along $ a $ to $ \rho $
(and similarly
for
$ (A,B) \mapsto (AB, B) $
and
$ b $).
Hence, to every
lamination
$ \lambda $
we associate
a sequence
$ (\varphi_n) $
in the mapping class group,
with the property that
$ \varphi_{n+1} \varphi_n^{-1} $
is a Dehn twist along $ a $ or $ b $.

\section{The spectrum is a Cantor set}\label{sec:cantor}

Let
$ c = (A,B) $
be a cocycle.
Recall that the spectrum
of
$ c $
is the set of IET
$ T $
such that
the Lyapunov exponent
$ \chi(c, T) $
is zero.
In this section,
we prove our main result
for cocycles:

\begin{theorem}\label{thm:main_spectrum}
	The spectrum of a generic cocycle
	is either empty,
	or a Cantor set.
\end{theorem}

Here,
we say that
a cocycle
is
\emph{generic}
if
it does not have a finite order rotation,
nor a parabolic, in the semigroup it generates.
The set of IETs is identified with
the interval $ \left[ 0,1 \right] $.
During the course of the proof of this result,
we will show the following result,
which characterizes completely the
behavior of the cocycle under renormalization.

\begin{theorem}\label{thm:main_renorm}
	Let $ c $ be a cocycle
	and $ T $ an IET.
	If $ T $ is not in the spectrum of $ c $,
	then $ (T, c) $ is uniformly hyperbolic.
	If $ T $ is in the spectrum of $ c $
	and $ T $ is not finite order,
	then the iterates of $ (T,c) $ under
	the accelerated renormalization
	stay bounded.
\end{theorem}

For a finite order IET,
the corresponding statement is easy:
if $ (A_n,B_n) $ is the last cocycle
of the sequence of renormalizations,
and if the last renormalization applied
was a $ \tau_1 $,
then the IET is in the spectrum if and only
if $ A_n $ is not hyperbolic
(respectively $ \tau_2 $
and $ B_n $).

Let us explain the principle of the proof
of these results.
Let $ T $ be 
a
$ 2 $-IET
and
$ (A,B) $ a cocycle above $ T $.
The accelerated Rauzy induction process
introduced in
Section~\ref{sec:renormalisation}
applied to $ T $
gives a sequence
$ (N_i) $
(infinite if
$ T $ is of infinite order),
which characterizes $ T $,
and a sequence of renormalized cocycles
$ (A_i,B_i) $,
related by
$ (A_{i+1},B_{i+1}) =
\tau_{\gamma_i}^{N_i}
(A_i,B_i) $,
where $ \gamma_i = 0,1 $
alternatively.

We categorize a pair $ (A,B) $ according to its \emph{type},
the pair $ (t_1,t_2) $,
where $ t_1,t_2 \in \left\{ \Hyp, \Ell \right\} $
are the type ($	\Hyp$ for hyperbolic or $ \Ell $ for elliptic) of $ A $ and $ B $.
We refine the type $ (\Hyp, \Hyp) $ by saying that the type of
$ (A,B) $ is $ (\Hyp, \Hyp)^{+} $,
if the axis of $ A $ and $ B $ point in the
same direction
and
$ (\Hyp, \Hyp)^{-} $ otherwise.

The core of the proof is the
description of the possible transitions between
the different types,
when we renormalize the pair $ (A,B) $.
These transitions are represented in
the diagram~\ref{fig:diag}.

\begin{figure}[h]
\[\begin{tikzcd}
	&& {} && \\
	&& {(\Ell,\Ell)} \\
	{(\Hyp,\Ell)} &&&& {(\Ell,\Hyp)} \\
		      && {\textcolor{Mahogany}{(\Hyp,\Hyp)^+}} \\
	&& {(\Hyp,\Hyp)^-}
	\arrow["{\exists^{\infty}\tau_2^N}", shift left=2, from=3-1, to=3-1, loop, in=55, out=125, distance=10mm]
	\arrow["{\exists^{\infty}\tau_1^N,\tau_2^N\textcolor{ForestGreen}{(K)}}", from=2-3, to=2-3, loop, in=55, out=125, distance=10mm]
	\arrow["{\exists^{\infty}\tau_1^N}", tail reversed, from=2-3, to=3-5]
	\arrow["{\exists^{\infty}\tau_2^N}", shift left=3, tail reversed, from=3-1, to=2-3]
	\arrow["{\exists^{\mathrm{fin}}\tau_1^N\textcolor{ForestGreen}{(K)}}"'{pos=0.2}, shift right, from=3-1, to=3-1, loop, in=305, out=235, distance=10mm]
	\arrow["{\forall^{\mathrm{l.e.}}\tau_1^N}"{pos=0.6}, curve={height=-6pt}, from=3-1, to=4-3]
	\arrow["{\exists^{\mathrm{fin}}\tau_2^N\textcolor{ForestGreen}{(K)}}"{pos=-0.15}, from=3-5, to=3-5, loop,
	in=305, out=235, distance=10mm, shift right=3]
	\arrow["{\forall^{\mathrm{l.e.}}\tau_2^N}"', curve={height=6pt}, from=3-5, to=4-3]
	\arrow["{\exists^{\mathrm{fin}}\tau_1^N}", curve={height=-3pt}, from=5-3, to=3-1]
	\arrow["{\exists^{\mathrm{fin}}\tau_2^N}"', shift right, curve={height=6pt}, from=5-3, to=3-5]
	\arrow[from=5-3, to=4-3]
	\arrow["{\exists^{\infty}\tau_1^N}", shift left=2, from=3-5, to=3-5, loop, in=55, out=125, distance=10mm]
\end{tikzcd}\]
  \caption{Transitions between types.}\label{fig:diag}
\end{figure}
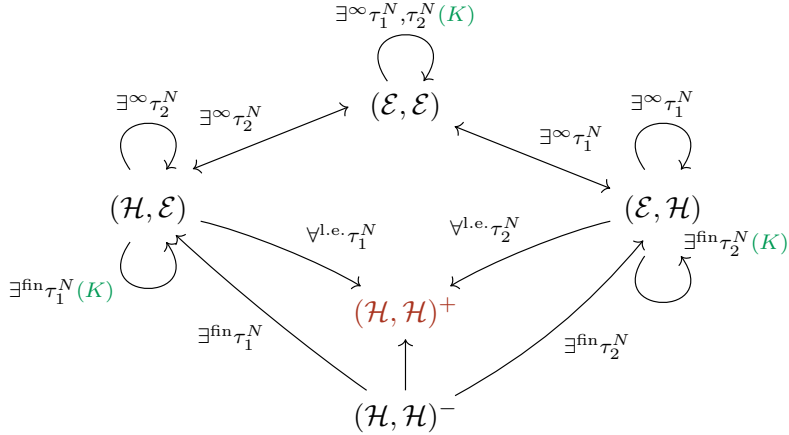

The diagram should be read as follows.
Each vertex is a type $ (t_1,t_2) $.
An edge between two types is labeled
with the type of renormalization,
$ \tau_1^N $ or $ \tau_2^{N} $, that is applied. 
Furthermore, for each transition,
we indicate the number of $ N $ that can be chosen
to give this transition:
finitely many
($ \exists^{\mathrm{fin}} $),
infinitely many
($ \exists^{\infty} $),
or for every $ N $ large enough
($ \forall^{\mathrm{l.e.}} $).
The justifications that these are exactly the possible
transitions are given in Section~\ref{sec:tech_lem}.

The sequence
$ (A_i,B_i) $ 
of renormalizations of the cocycle
over the IET $ T $ gives a path
in this diagram,
with the property that a transition
of type $ \tau_1^{N} $ must be followed
by a transition of type $ \tau_2^{N} $
(and vice versa).
Reciprocally,
every such path
(finite or infinite)
gives a sequence
$ (N_i) $,
which corresponds to a unique IET.

In order to decide that a trajectory
is in the spectrum,
we rely on a classical argument
which is
that
if
the iterated renormalizations of a cocycle
belong to a common compact set,
then
the Lyapunov exponent of the cocycle
vanishes.
The precise statement is
Proposition~\ref{prop:bounded_renom}
and is proved in
Section~\ref{subsec:zeroLyap}.


In our case,
the compact set will be
the subset
$ K $ 
of
the relative character variety defined by
the fact that at least two matrices between
$ A,B $ and $ AB $ are elliptic:

\begin{lemma}~\label{lem:peage_K}
	The set $ K $ is compact,
	and
	\begin{enumerate}
		\item 
			If a pair is of type $ (\Ell, \Ell) $,
			then it is in $ K $.
		\item 
			If a pair
			$ (A,B) $ 
			is of type
			$ (\Hyp, \Ell) $,
			and such that its renormalization by
			$ \tau_1^{N} $
			is still
			$ (\Hyp, \Ell) $,
			then
			$ \tau_1^{k} (A,B) $
			is in $ K $
			for $ k=0, \dots, N-1 $.
			In particular,
			$ \tau_1^{N} (A,B) $
			is in
			$ \tau_1 K $.
	\end{enumerate}
\end{lemma}
This lemma explains the annotation $(K)$ on the diagram.

\begin{proof}
	Recall that the relative character variety
	$ \cv_c $ 
	is defined by the Markov equation
	\begin{equation}
		x^2+y^2+z^2
		-xyz
		= c + 2,
	\end{equation}
	where
	$ x = \tr A $,
	$ y = \tr B $,
	$ z = \tr AB $,
	and
	$ c = \tr \left[ A,B \right] $.
	If at least 2 matrices between
	$ A,B $ and $ AB $ are elliptic,
	for example $A$ and $B$,
	then $ x,y \in \left[ -2,2 \right] $
	and as $ z $ satisfies the Markov equation,
	it must be bounded.
	This proves that $ K $ is compact.

	By definition, $ (\Ell,\Ell) $ is in $ K $.
	If $ (A,B) $ is $ (\Hyp,\Ell) $
	and a renormalization
	$ \tau_1^{N} $ is still in
	$ (\Hyp, \Ell) $,
	then by Proposition~\ref{prop:from_HE},
	$ \tau_1^{k} (A,B) = (A, BA^{k}) $
	is also
	$ (\Hyp, \Ell) $,
	for $ k=0, \dots, N $.
	This means that
	$BA^k$ and $ (BA^{k})A $ are elliptic,
	for $ k=0, \dots, N-1 $,
	hence that
	$ (A,B A^{k}) $ is in $ K $.

\end{proof}

The symmetric statement holds for the type
$ (\Ell, \Hyp) $.

\subsection{Proof of the results}\label{sub:proof_main}
We now prove
Theorem~\ref{thm:main_spectrum}
and
Theorem~\ref{thm:main_renorm}.
To prove that the spectrum of a cocycle
is a Cantor set,
we use the classical criterion:
a subset of
$ \left[ 0,1 \right] $
is a Cantor set if and only if
it is
closed, without isolated points,
and with empty interior.

\vspace{2mm}

As in Section~\ref{sec:renormalisation},
we identify a 2-IET in
$ \left[ 0,1 \right] $
with its sequence
$ (\gamma_n)$
of top and bottom.
Starting from the cocycle
$ (A,B) $,
we have the associated sequence
of renormalizations
$ (A_n,B_n) $.
This sequence gives us a path
in the main diagram~\ref{fig:diag},
which follows the types of the
renormalizations.
Let $ S $ be the set of sequences
$ (\gamma_n) $
such that the associated
$ (A_n,B_n) $ 
never lands in $ (H,H)^{+} $.
We first show that $ S $,
if not empty,
is  a Cantor set,
then that $ S $ coincides with the spectrum of
our cocycle.

\vspace{2mm}

\emph{$S$ is closed.}
If a sequence is not in $ S $,
then it lands in $ (H,H)^{+} $
at some stage.
Any other sequence
which is
close enough
to it share the same initial
trajectory up to this stage,
hence also lands in $ (H,H)^+ $,
and is also not in $ S $.

\vspace{2mm}

\emph{$S$ has empty interior.}
Consider a sequence in $ S $.
At any given stage,
we can make it bifurcate so that it
lands in $ (H,H)^+ $:
this proves that the complement
of $ S $ is dense.
This is established by a direct
inspection of the diagram.
If we are in $ (E,E) $,
we can assume that we
come from $ \tau_1 $ by symmetry,
then we can go to $ (H,E) $ by a
$ \tau_2^{N} $
and then to $ (H,H)^{+} $ by
a $ \tau_1^{N} $.
If we are in
$ (H,E) $,
coming from $ \tau_2 $,
we can apply a
$ \tau_1^N $ to go to
$ (H,H)^{+} $.
If we came from $ \tau_1 $,
we can apply a
$ \tau_2^{N} $ to stay
in $ (H,E) $,
and we are back in the previous case.
The other cases are dealt with
in the same way.

\vspace{2mm}

\emph{$S$ has no isolated points.}
Consider a sequence in $ S $.
As in the previous paragraph,
we can make it bifurcate at any stage,
so that it is different, but stays in $ S $:
If at some stage it is in $ (E,E) $,
we can create a subsequent different sequence
that stays forever in $ (E,E) $, by alternating
some $ \tau_1^N $ and $ \tau_2^N $.
If at some stage it is in
$ (H,E) $,
then in at most 2 steps we can
reach $ (E,E) $,
and apply the previous case.

\emph{$S$ coincides with the spectrum.}
If a sequence is not in $ S $,
then it is $ (H,H)^{+} $ at some stage,
hence uniformly hyperbolic by
Proposition~\ref{prop:uh},
and not in the spectrum.
If a sequence is in $ S $,
we show that the sequence
of renormalizations of the cocycle
is bounded,
hence that the IET is in the spectrum
by Proposition~\ref{prop:bounded_renom}.
After some stage,
the sequence lies only in
$ (E,E) $, $ (H,E) $
or $ (E,H) $.
When in $ (E,E) $,
the cocycle is in $ K $
by Lemma~\ref{lem:peage_K}.
When in $ (H,E) $,
there are two cases.
First case,
the next step is
$ \tau_1^{N} $ to stay
in $ (H,E) $,
and the cocycle is in $ K $,
again by Lemma~\ref{lem:peage_K}.
Second case,
the next step is
$ \tau_2^{N} $,
but this means that the previous
step was a $ \tau_1^{N} $ coming
from $ (H,E) $.
Hence we can write
our cocycle
$ c_1 = \tau_1^{N} c_0 $
with $ c_0 $ in $ (H,E) $,
and by Lemma~\ref{lem:peage_K},
all the $ \tau_1^{k} c_0 $,
for $ k=0,\dots,N $, are in $ (H,E) $.
Hence $ \tau_1^{N-1} c_0 $ is in $ K $
by Lemma~\ref{lem:peage_K},
and $ c_1 $ is in $ \tau_1(K) $.
The case $ (E,H) $ is similar,
and we have shown that at all stages,
the renormalized cocycles are in
the compact
$ K \cup \tau_1(K) $.
This concludes the proof.

\subsection{Bounded renormalizations imply zero Lyapunov exponent}\label{subsec:zeroLyap}
Here we explain how an argument classically used to control Birkhoff sums over rotations and interval exchange transformation can be adapted to cocycles, as to control the Lyapunov exponent of a cocycle from the knowledge of its behaviour under renormalization.

\vspace{2mm}

Let $T_0 : [0,1] \longrightarrow [0,1]$ be a $2$-IET and let $(A_0,B_0)$ be a pair of matrices in $\mathrm{SL}(2, \mathbb{R})$. We denote by $\rho_0$ the associated piecewise constant cocycle over $T_0$. In what follows we denote by 

\begin{itemize}

\item $I_n = [0,a_n]$ the interval such that the $n$-th (accelerated) renormalization of $T_0$ is the first-return of $T_0$ on $I_n$;

\item $T_n : I_n \longrightarrow I_n$ the $n$-th (accelerated) renormalization of $T_0$;

\item $(A_n,B_n)$ the $n$-th (accelerated) renormalization of the pair of matrices $(A_0,B_0)$ over $T_0$ and $\rho_n$ the associated cocycle over $T_n$.

\end{itemize}

\begin{proposition}\label{prop:bounded_renom}
	Assume that the sequence $(A_n,B_n)$ defines a sequence of representation
	$\rho_n$ which remain within a compact subset of
	$\cv(\pi_1 \mathrm{T}^*, \mathrm{SL}(2,\mathbb{R}))$. Then for all $x \in [0,1]$,
	$$\frac{1}{n} \log ||\rho(T_0^{n-1}(x)) \cdots \rho(T_0(x)) \cdot \rho(x)|| \rightarrow 0.$$
	This implies in particular that 
	$$ \chi(\rho, T_0) = 0.$$

\end{proposition}

The rest of the paragraph is dedicated to the proof of this Proposition.

\vspace{2mm}

\paragraph*{\bf Decomposition of the orbit}

We start with $n > 0$ which we think of as large enough.

\vspace{2mm}

The (classical) idea is to decompose the orbit of $x$ under $T_0$ using times corresponding to times of the sequence of renormalization $T_n$. Pick $m$ a (large) integer such that at least two points of the orbit of $x$ up to time $n$ belong to $I_m$. Let $n_0, n_1$ be the two integers such that 

\begin{enumerate}

\item $n_0$ is the smallest integer such that $T_0^{n_0}(x) \in I_m$;

\item $n_1$ is the smallest integer such that $T_0^{n-n_1}(x) \in I_m$.
\end{enumerate}

We start by observing that there exists a uniform constant $C_M > 0$ such that for all $x$, for all $n$ as above we have $n_0, n_1 \leq C_M$. Since $T_m$ is the first return of $T_0$ on $I_m$, we have 

$$ T_0^n(x) = T_0^{n_1} \circ T_m^{k} \circ  T_0^{n_0}(x) $$ for some integer $k$. 

\vspace{2mm}

\paragraph*{\bf Bounds for the cocycle} By hypothesis, the pair of matrices
$(A_m,B_m)$ induces a representation $\rho_m$ which remains in a fixed compact
set of the character variety.
This means that there exists a compact set
$K \subset
\mathrm{SL}(2,\mathbb{R}) \times  \mathrm{SL}(2,\mathbb{R})$,
and for all $m \in \mathbb{N}$ a matrix $P_m \in \mathrm{SL}(2,\mathbb{R})$ such
that $P_m(A_m,B_m)P_m^{-1} = (A'_m, B'_m) \in K$. In particular there exists a
constant $D > 0$, which does not depend on $m$, such that $||A'_m||, ||B'_m|| \leq D $.

\vspace{2mm}

\paragraph*{\bf Estimating the product of matrices.} Recall that for any $m \in
\mathbb{N}$, $A_m =   \rho(T_0^{k_m}(y)) \cdots \rho(T_0(y))\cdot\rho(y)$ for a
point $y$ in the left continuity interval of $T_m$ and where $k_m$ is the
integer such that $T_m(y) = T_0^{k_m}ty)$ on this continuity interval.
Furthermore, we have seen that $k_m$ is equal to $q_m$ or $q_m + q_{m-1}$, by Proposition \ref{prop:returntimes}.

\vspace{2mm} From $ T_0^n(x) = T_0^{n_1} \circ T_m^{k} \circ  T_0^{n_0}(x) $ we get 
$$ \rho(T_0^{n-1}(x)) \cdots \rho(T_0(x)) \cdot \rho(x) = M_1 \cdot \omega(A_m,B_m) \cdot M_0 $$ where 

\begin{itemize}
	\item $M_0$ is a product of $n_0$ matrices which are either $A_0$ or $B_0$;
	\item $M_1$ is a product of $n_1$ matrices which are either $A_0$ or $B_0$;
	\item $\omega(A_m,B_m)$ is a word in $A_m$ or $B_m$;
\end{itemize}

Writing $\omega(A_m,B_m) = P_m^{-1} \omega(A'_m,B'_m) P_m$, we get 
$$ || P_m^{-1} \omega(A'_m,B'_m) P_m || \leq || P_m^{-1} || || P_m || D^k.$$
Furthermore, we can bound $|| M_0 || , ||M_1|| \leq E^C$ where $E = \sup(||A_0||, ||B_0||)$.
This yields 
$$ ||\rho(T_0^{n-1}(x)) \cdots \rho(T_0(x)) \cdot \rho(x) || \leq || E^2 P_m^{-1} || || P_m || D^k.$$
Passing to the logarithm we get 
$$ \frac{1}{n} \log ||\rho(T_0^{n-1}(x)) \cdots \rho(T_0(x)) \cdot \rho(x) || \leq  \frac{F_m}{n} + \log D \frac{k}{n}.$$
Since $n \geq k q_m$ we get
$$ \frac{1}{n} \log ||\rho(T_0^{n-1}(x)) \cdots \rho(T_0(x)) \cdot \rho(x) || \leq  \frac{F_m}{n} + \log D \frac{1}{q_m}.$$
So making $n$ tend to infinity we obtain
 $$\limsup \frac{1}{n} \log ||\rho(T_0^{n-1}(x)) \cdots \rho(T_0(x)) \cdot \rho(x)|| \leq \log D \frac{1}{q_m}.$$
 This is independent of the choice of $m$, and since $q_m$ tends to infinity this concludes the proof.

\section{Back to representations}\label{sec:back_to_rep}

In this section,
we translate the statements on cocycles and renormalization established in the previous sections
to the corresponding ones about representations and the action
of the mapping class group.

First, we describe the spectrum of a generic representation.
\begin{theorem}\label{thm:repr_spectrum}
	Let $ \rho \in \cv_c (\Ts, \SL(2,\R)) $
	be a generic representation.
	If $ \rho $ is not simple Fuchsian,
	then its spectrum is non-empty,
	and is a
	Cantor set.
\end{theorem}

\begin{proof}
	According to Section~\ref{sec:cantor},
	the Lyapunov exponent of
	$ \rho $ with respect to a measured lamination
	$ \lambda $ is the Lyapunov exponent of
	the induced locally constant cocycle above
	a 2-IET.
	Depending on the slope of the lamination,
	the cocycle is
	of the form
	$ (A^{\pm} ,B^{\pm} ) $,
	where
	$ (A,B) = \left( \rho(a), \rho(b) \right) $.
	The spectrum of $ \rho $ is then
	the union of the spectrum of the cocycles
	$ (A^{\pm} ,B^{\pm} ) $,
	and by
	Theorem~\ref{thm:main_spectrum}
	each of these spectrum is empty or a Cantor set.
	It remains to prove that
	if $ \rho $ is not simple Fuchsian,
	at least one of them is not empty.
	By contraposition,
	we suppose that the spectra of
	$ (A^{\pm} ,B^{\pm} ) $
	are empty.

	The type of $ (A,B) $
	must be
	$ (\Hyp,\Hyp)^{+} $
	or
	$ (\Hyp,\Hyp)^{-} $.
	If it is
	$ (\Hyp,\Hyp)^{-} $,
	in the terminology of
	Lemma~\ref{lem:prodHH},
	it must be in the first or third case,
	which means that
	$ \tau_1(A,B) = (A,AB) $ is either
	$ (\Hyp,\Hyp)^{-} $
	or
	$ (\Hyp,\Hyp)^{+} $.
	If moreover
	$ (A,AB) $ is
	$ (\Hyp,\Hyp)^{+} $
	(see
	Figure~\ref{fig:fucshian})
	then it is classical that
	$ \rho $
	is the holonomy of a hyperbolic structure
	on a pair of pants
	(see e.g.~\cite{Goldman}),
	hence is simple Fuchsian.
	Arguing by contradiction,
	we assume that
	we are not in this case,
	so that
	$ (A^{\pm} ,B^{\pm} ) $,
	and its orbit under the mapping class group,
	are either
	$ (\Hyp,\Hyp)^{+} $
	or 
	the first case of
	$ (\Hyp,\Hyp)^{-} $.

	Up to replacing
	$ (A,B) $
	by
	$ (A,B^{-1}) $,
	we can assume that
	it is in the first case of
	$ (\Hyp,\Hyp)^{-} $
	(recall that, as we assume
	$ c > 2 $,
	the axis of $ A $ and
	$ B $ do not intersect).
	We then have
	that
	$ \tau_1 (A,B) = (A,AB) $ is
	$ (\Hyp,\Hyp)^{-} $,
	and the proof of
	Lemma~\ref{lem:prodHH}
	shows that
	$ \ell_{AB} \le \ell_{B} - \ell_{A} $
	(denoting by $ \ell_A $
	translation length of a hyperbolic isometry
	$ A $).
	By assumption,
	$ \tau_1 (A,B) $
	is still in the first case of
	$ (\Hyp,\Hyp)^{-} $,
	and,
	iterating this argument,
	so is
	$ \tau_1^n (A,B) $
	for every $ n > 0 $.
	We then have
	$ \ell_{A^n B} \le \ell_{B} - n \ell_{A} $,
	which is a contradiction.
\end{proof}

We record here the following 
corrollary:
\begin{corollary}
	If $ \rho $ is not simple Fuchsian,
	then its spectrum is not empty,
	and contains a simple curve.
\end{corollary}
\begin{proof}
	The fact that the spectrum is not empty is
	contained in the previous Theorem,
	and the proof of
	Theorem~\ref{thm:main_spectrum}
	shows that when the spectrum is not empty,
	it contains a simple curve.
\end{proof}
Note that the second point is a result of Goldman~\cite{Goldman}.

We now state the corresponding result for the dynamics of
the mapping class group.
Recall that in Section~\ref{sub:traj_mcg}
we associated to every measured lamination
a trajectory in the mapping class group,
corresponding to the sequence of renormalizations.
\begin{theorem}\label{thm:repr_mcg}
	Let $ \rho \in \cv_c (\Ts, \SL(2,\R)) $
	be a generic representation
	and
	$ \lambda $ a measured lamination.
	Let $ (\varphi_n) $ be the trajectory in
	the mapping class group associated to $ \lambda $.
	\begin{itemize}
		\item 
			If $ \lambda $ is not in the spectrum of $ \rho $,
			then $ \varphi_n \rho $ goes to infinity in the
			character variety.
		\item 
			If $ \lambda $ is in the spectrum of $ \rho $,
			then $ \varphi_n \rho $ does not go to infinity.
	\end{itemize}
\end{theorem}

As the proof will show,
we can be more precise in both cases.
If $ \lambda $ is not in the spectrum of $ \rho $,
and its support is not a simple curve,
then
$ \varphi_n \rho $ goes to infinity
exponentially fast:
there exist a constant $ C > 0 $
such that for $ n $ large enough:
\begin{equation}
	\norm{\varphi_n \rho}
	\ge
	e^{C n}
	,
\end{equation}
where
we consider $ \cv_t $ as
the Markov surface
embedded
in $ \R^{3} $,
and
$ \norm{\cdot} $ is any norm.
Moreover,
with $ k_n = a_1 + \cdots + a_n $ being the times
corresponding
to the accelerated renormalization
(in the notations of Section~\ref{sec:renormalisation}),
we have
\begin{equation}
	\norm{\varphi_{k_n}  \rho}
	\ge
	e^{C q_n}
	,
\end{equation}
and recall that
$ q_n \ge (\sqrt{2})^{n} $,
so that the growth is \emph{super-exponential}
at those times.

Reciprocally, if $ \lambda $ is in the spectrum of $ \rho $,
then there exists a compact $ K $,
which does not depend on $ \rho $ or $ \lambda $,
such that
$ \varphi_{k_n} \rho $ is in $ K $.
In particular, if the continued fraction development
associated to $ \lambda $ is of \emph{bounded type},
then $ \varphi_n \rho $ stays in a compact.

We also remark that
by the classical relationship between
$ \SL(2,\Z) $
and
continued fraction,
every diverging sequence
in
$ \SL(2,\Z) $
is at a bounded distance
from a subsequence
of the
$ (\varphi_n) $
associated to some
lamination $ \lambda $.
In particular,
if $ (f_n) $ is a
diverging
sequence
of hyperbolic matrices in
$ \SL(2,\Z) $,
and the sequence of repulsive fixed lines
of $ (f_n) $ converges
to $ \lambda $,
then the conclusion of
Theorem~\ref{thm:repr_mcg} applies to
$ (f_n) $.

\begin{proof}
	First,
	we assume that
	$ \lambda $ is not in the spectrum of $ \rho $.
	By the proof of
	Theorem~\ref{thm:main_spectrum},
	there exists a $ n_0 $ such that
	$ (A_0, B_0) = (\varphi_{n_0} \rho)(a,b) $
	is of type
	$ (\Hyp, \Hyp)^+ $
	(up to replacing $ a $ or $ b $
	by its inverse).
	By (the proof of) Proposition~\ref{prop:uh},
	the semigroup generated by
	$ (A_0,B_0) $ is uniformly hyperbolic.
	In particular, there exists a constant
	$ C > 0 $ such that
	\begin{equation}
		r( w(A_0,B_0) ) \ge C r(A_0)^{\abs{w}_1} r(B_0)^{\abs{w}_2},
	\end{equation}
	where
	$ r $ is the spectral radius of a matrix,
	$ w $ is a word in the semigroup generated by $ A_0,B_0 $,
	$ \abs{w}_1 $ and $ \abs{w}_2 $  are the number of occurrences of $ A_0 $
	and $ B_0 $ in $ w $.

	We write
	$ (A_n, B_n) =
	\varphi_n (A_0,B_0) $.
	The total number of $ A_0 $
	and $ B_0 $
	in
	$ A_n$
	and
	$ B_n $,
	$ \abs{A_n}_1
	+
	\abs{A_n}_2
	+
	\abs{B_n}_1
	+
	\abs{B_n}_2 $,
	is equal to 
	the sum of the coefficients
	of the matrix in $ \SL(2,\Z) $
	representing $ \varphi_n $.
	Denote by
	$ \norm{\varphi_n}_1 $
	this quantity.
	By the definition of
	$ \varphi_n $,
	$ \norm{\varphi_n}_1 \ge n $.
	Taking logarithms in the previous
	equation,
	we then have:
	\begin{equation}
		\log r (A_n)
		+
		\log r (B_n)
		\ge
		n
		\min \left( 
			\log r(A_0),
			\log r(B_0)
		\right)
		+
		2 \log C,
	\end{equation}
	and up to renaming the constant we have
	\begin{equation}
		\log ( r (A_n) + r (B_n) )
		\ge
		C n.
	\end{equation}
	Recall that the coordinates of the embedding
	of the character variety $ \cv(\Ts, \SL(2,\R)) $
	in $ \R^3 $ are
	$ x = \tr A, y= \tr B$
	and
	$z=\tr AB $.
	We equip $ \R^3 $ with the norm
	$ \norm{(x,y,z)} = \abs{x} + \abs{y} + \abs{z} $,
	and as we have
	$ \abs{\tr M} \ge r(M) $
	when $ M $ is hyperbolic,
	we conclude that
	\begin{equation}
		\norm{ \varphi_n (A_0,B_0) }
		\ge
		e^{C n} 
		.
	\end{equation}
	For the estimate at the times
	of the accelerated renormalizations
	$ k_n $,
	we argue in the same way,
	using the fact that
	$ \norm{\varphi_{k_n}}_1 \ge q_n $,
	(see
	Section~\ref{sec:renormalisation}).

	We now assume that
	$ \lambda $ is in the spectrum of
	$ \rho $.
	Considering the associated
	cocycle $ (A,B) $,
	Theorem~\ref{thm:main_renorm}
	asserts that
	the sequence of accelerated renormalizations
	of $ (A,B) $ stays in a universal compact
	$ K $  of the character variety.
	Translated back to $ \rho $,
	this means that
	$ \varphi_n \rho $
	stays in $ K $ at the times
	$ k_n $
	of the acceleration.
	In particular,
	it does not diverge.
	If $ \lambda $ is of bounded type,
	this means that the differences
	$ k_{n+1} - k_n $ are bounded,
	and the excursions outside
	$ K $ of
	$ \varphi_n \rho $
	must also be bounded.
\end{proof}

\section{Hyperbolic geometry warm-up}\label{sec:hyp_geom}

In what follows, $A$ and $B$ are two elements in $\mathrm{PSL}(2, \mathbb{R})$. In this paragraph, we systematically answer the question of determining the type of $A \cdot B$ as an isometry of the hyperbolic space $\mathbb{H}$.
All the statements of this section are elementary exercises in hyperbolic
geometry: proof are either briefly sketched or left to the
reader, with the hope that the pictures are convincing enough.

\vspace{2mm} \fbox{
\parbox{\textwidth}{ 
The hyperbolic space $\mathbb{H}$ is given an orientation, and all angles are \textit{oriented} angles which we will think of as elements of $\mathbb{R}/ 2\pi \mathbb{Z}$.
}}

\vspace{2mm}

\subsection{Preliminaries}

We first recall basic facts about isometries of $\mathbb{H}$. Specifically, any orientation-preserving isometry of $\mathbb{H}$ can be written as the product of two reflections. We elaborate a bit on this fact here.

\begin{lemma}
\label{lem:decompR}
Let $A \in \mathrm{PSL}(2, \mathbb{R})$ be a rotation of angle $\alpha$ and centre $c$. Given $L_1$  a line through $c$, denote by $L_2$ the line through $c$ forming an (oriented) angle $\frac{\alpha}{2}$ with $L_1$ (\textit{i.e.} the oriented angle from $L_1$ to $L_2$ is $\frac{\alpha}{2}$). Then we have 

$$ A = \mathcal{R}_{L_2} \circ  \mathcal{R}_{L_1} $$ where  $\mathcal{R}_{L_1}$ and $ \mathcal{R}_{L_2}$ are the reflection through $L_1$ and $L_2$ respectively.

\end{lemma}


\begin{figure}[h]
	  \centering
	\begin{subfigure}{0.4\textwidth}
		  \centering
		\includegraphics[scale=0.3]{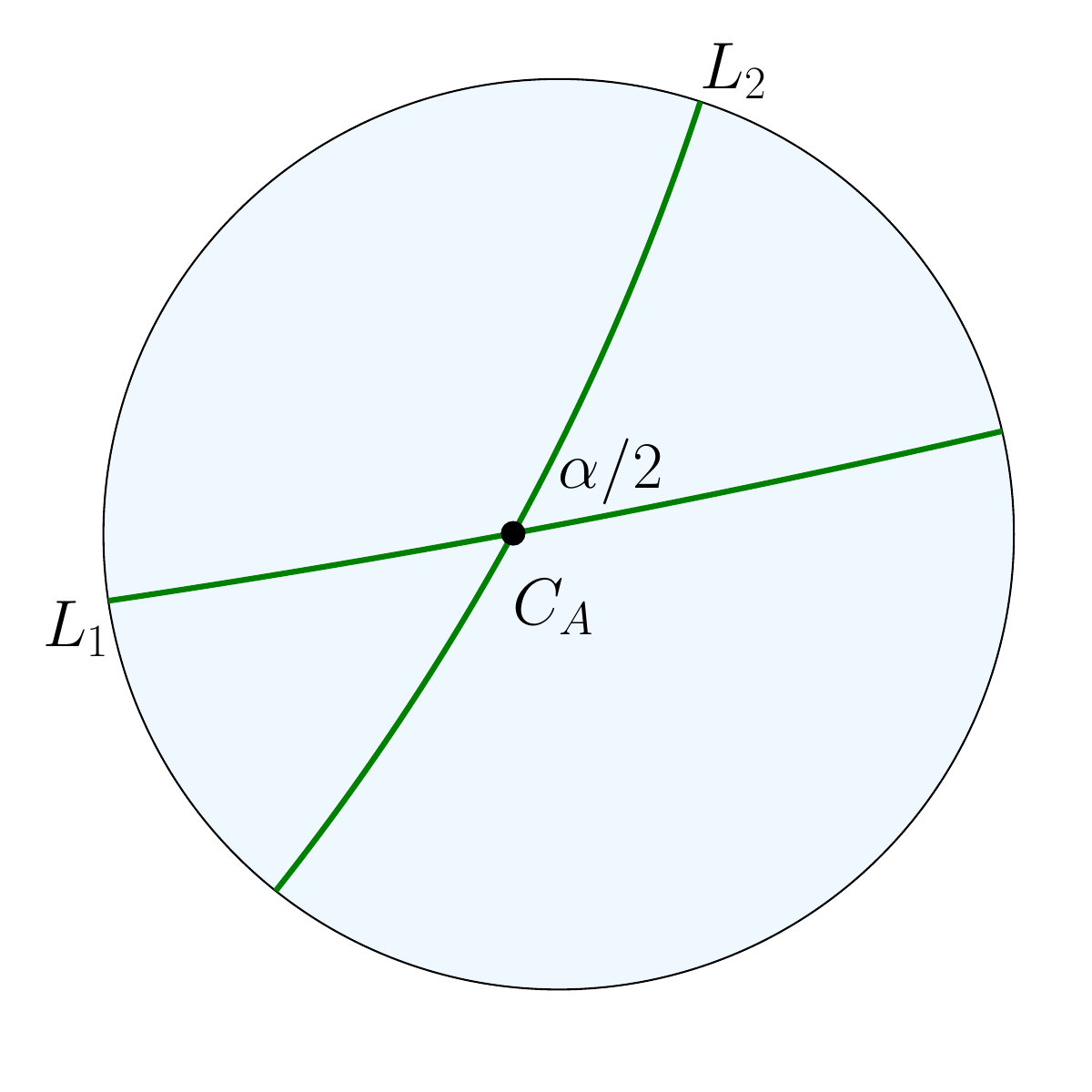}
		\caption{A rotation as a product of 2 reflections}
	\end{subfigure}
	\begin{subfigure}{0.4\textwidth}
		  \centering
		\includegraphics[scale=0.3]{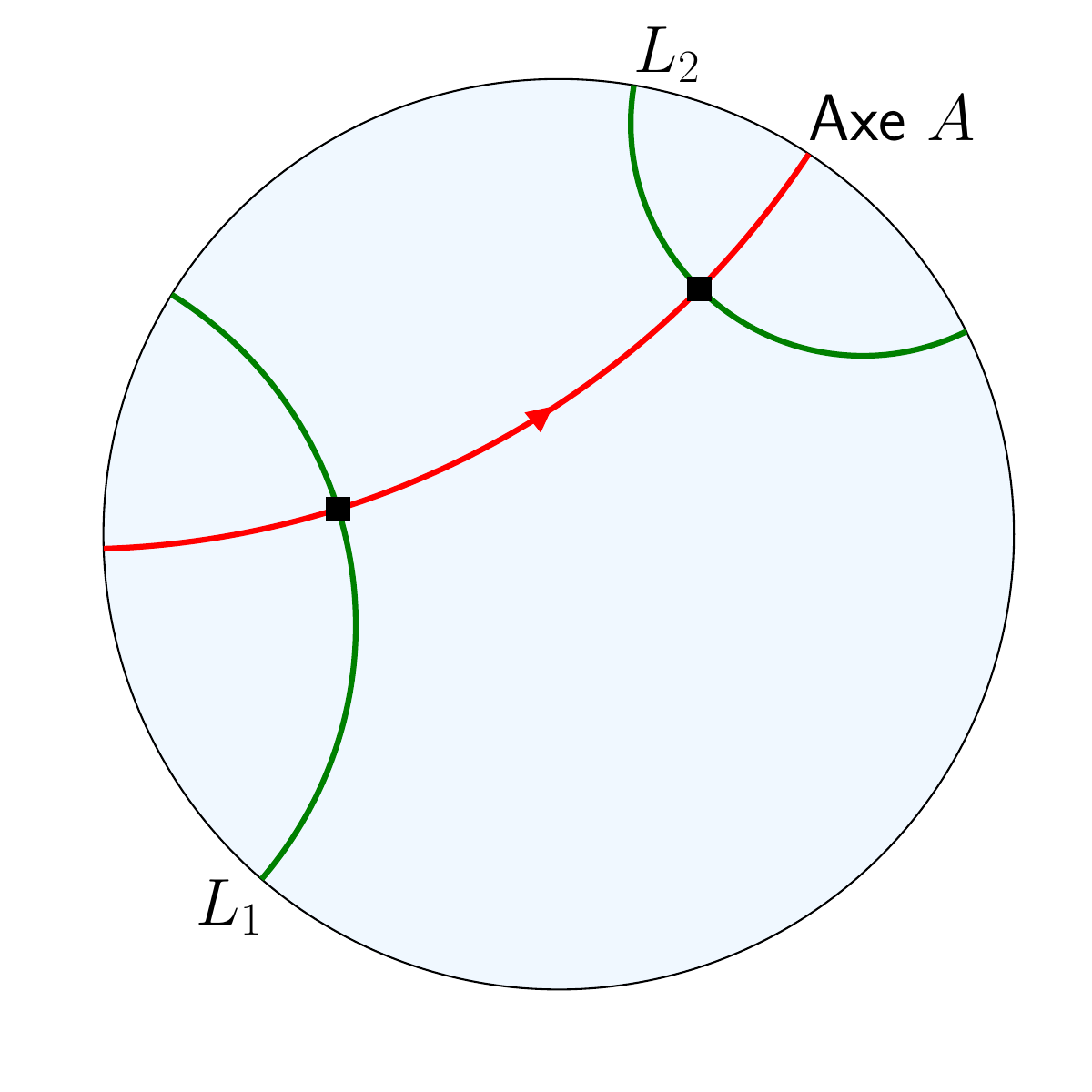}
		\caption{A translation as a product of 2 reflections}
	\end{subfigure}
\end{figure}

\begin{lemma}
\label{lem:decompH}
Let $A \in \mathrm{PSL}(2, \mathbb{R})$ be a hyperbolic translation of length $l >0$ and axis $L$. Given $L_1$ a line perpendicular to $L$ intersecting $L$ in a point $p$, denote by $L_2$ the line perpendicular to $L$ through the middle point of the line segment $[p,A(p)]$.  Then we have 

$$ A = \mathcal{R}_{L_2} \circ  \mathcal{R}_{L_1} $$ where $\mathcal{R}_{L_1}$ and $ \mathcal{R}_{L_2}$ are the reflection through $L_1$ and $L_2$ respectively.

\end{lemma}


Finally, we need some sort of converse to the previous two lemmata. 

\begin{lemma}
\label{lem:decomptype}
Let $L$ and $M$ be two lines in $\mathbb{H}$. The following statements hold.

\begin{enumerate}
\item If $L$ and $M$ intersect in $\mathbb{H}$, then $\mathcal{R}_{L} \circ \mathcal{R}_{M}$ is a rotation (of oriented angle half the angle between $L$ and $M$).
\item If $L$ and $M$ do not intersect in $\mathbb{H}$ but intersect in $\partial \mathbb{H}$, then $\mathcal{R}_{L} \circ \mathcal{R}_{M}$ is a parabolic isometry.
\item If $L$ and $M$ do not intersect in $\mathbb{H}$ nor in $\partial \mathbb{H}$, then $\mathcal{R}_{L} \circ \mathcal{R}_{M}$ is a hyperbolic isometry, of translation length twice the distance between $L$ and $M$.
\end{enumerate}

\end{lemma}


\subsection{Product of two rotations}

In this paragraph we assume that

\begin{itemize}
\item $A$ is a rotation of centre $c_A \in \mathbb{H}$ and of angle  $\theta_A \in [0,2\pi)$;
\item  $B$ is a rotation of centre $c_B \in \mathbb{H}$ and of angle  $\theta_B \in [0,2\pi)$;
\item $d$ is the distance between $c_A$ and $c_B$.
\end{itemize}

\noindent Note that the triple $(\theta_A, \theta_B, d)$ determines the pair $(A,B)$ up to conjugacy in $\mathrm{PSL}(2,\mathbb{R})$.

\begin{lemma}
\label{lem:prodRR}
Let $d$ and $\theta_A$ be given. Then there exists a positive number $\alpha = \alpha(d, \theta_A) < 2\pi$ such that we have:

\begin{enumerate}
\item for all $\theta_B \in (-\alpha, \alpha)$; $AB$ is elliptic;
\item for all $\theta_B \in \{-\alpha, \alpha \}$; $AB$ is parabolic;
\item for all $\theta_B \in (\alpha, 2\pi- \alpha)$; $AB$ is hyperbolic.
\end{enumerate}

\noindent Moreover, for fixed $\theta_A$, $\alpha$ tends to $0$ when $d$ tends to infinity.

\end{lemma}

\begin{proof}

Let $L$ be the line through $c_A$ and $c_B$. From Lemma \ref{lem:decompR}, there are two lines $L_A$ and $L_B$ through $c_A$ and $c_B$ respectively such that
\begin{itemize}

\item $A = \mathcal{R}_{L_A} \circ \mathcal{R}_{L}$ and $B = \mathcal{R}_{L} \circ \mathcal{R}_{B}$;

\item the lines $L$ and $L_A$ meet at an angle $\frac{\theta_A}{2}$ and $L$ and $L_B$ meet at an angle $\frac{\theta_B}{2}$, when suitably oriented.
\end{itemize}

Thus $AB = \mathcal{R}_{L_A} \circ \mathcal{R}_{L_B}$. Recall that a product of two reflections is 

\begin{enumerate}
\item a rotation if their axes meet in $\mathbb{H}$;
\item a parabolic element if their axes do not meet but share an end point in $\partial \mathbb{H}$;
\item a hyperbolic element if their axes do not meet and share no end point in $\partial \mathbb{H}$.
\end{enumerate}

\noindent If one fixes $d$ and $\theta_B$, the lines $L$ and $L_B$ are fixed.  Varying $\theta_A$ in the above construction corresponds to varying $L_A$. Determining the type of $AB$ amounts to working out the angles for which $L_A$ and $L_B$ intersect. 

\vspace{2mm} Since the line $L_A$ is through $c_A$, it is uniquely determined by one of its end points. One sees that $L_A$ and $L_B$ intersect if and only if one of the end points of $L_A$ belongs to one of two intervals in $\partial \mathbb{H}$ cut out by the end points of $L_B$. The Proposition follows from this remark, as varying the angle of $A$ amounts to varying one of the end points of $L_A$.

\end{proof}

\begin{figure}[h]
	  \centering
	\begin{subfigure}{0.4\textwidth}
		  \centering
		\includegraphics[scale=0.3]{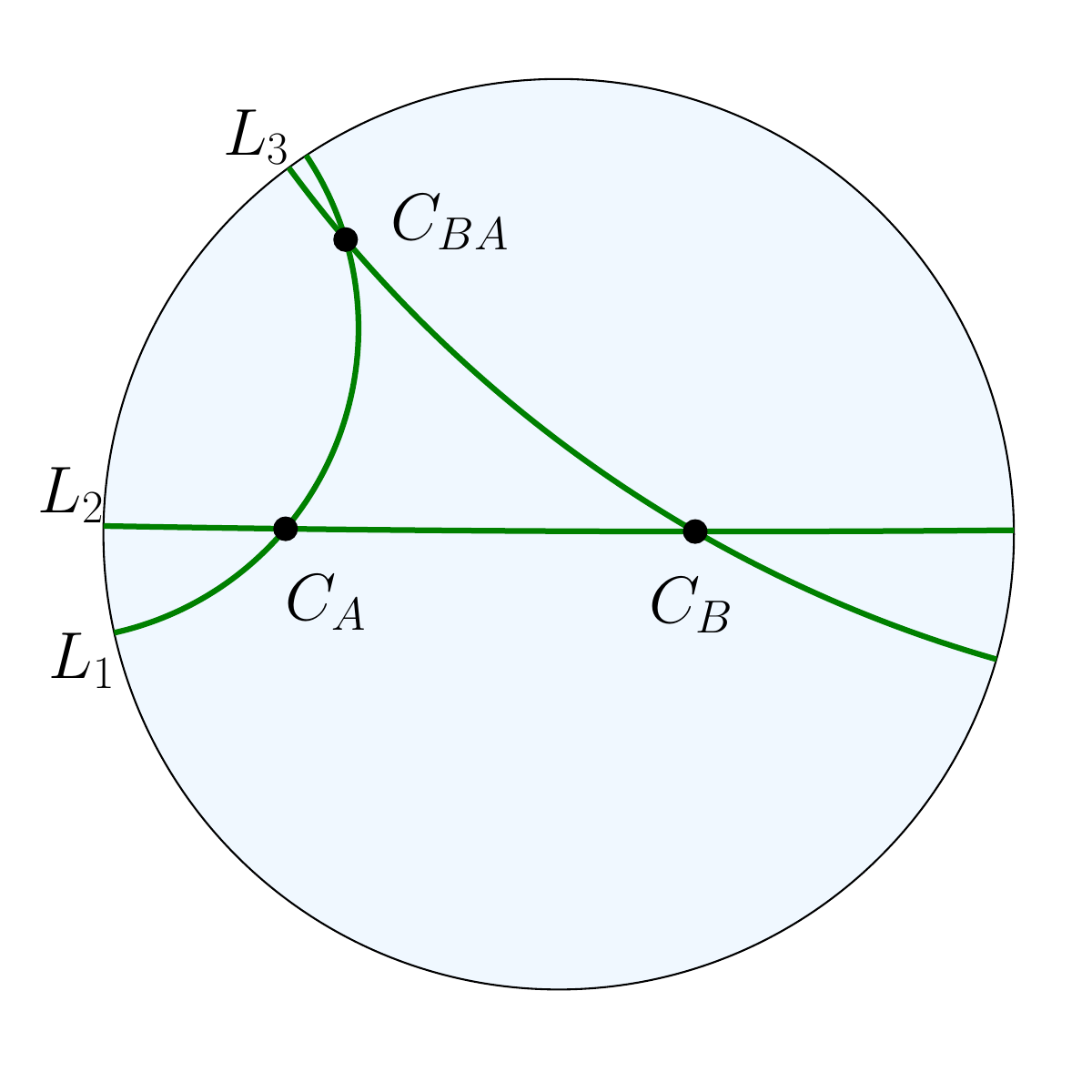}
		\caption{A product of two elliptic that is elliptic}
	\end{subfigure}
	\begin{subfigure}{0.4\textwidth}
		  \centering
		\includegraphics[scale=0.3]{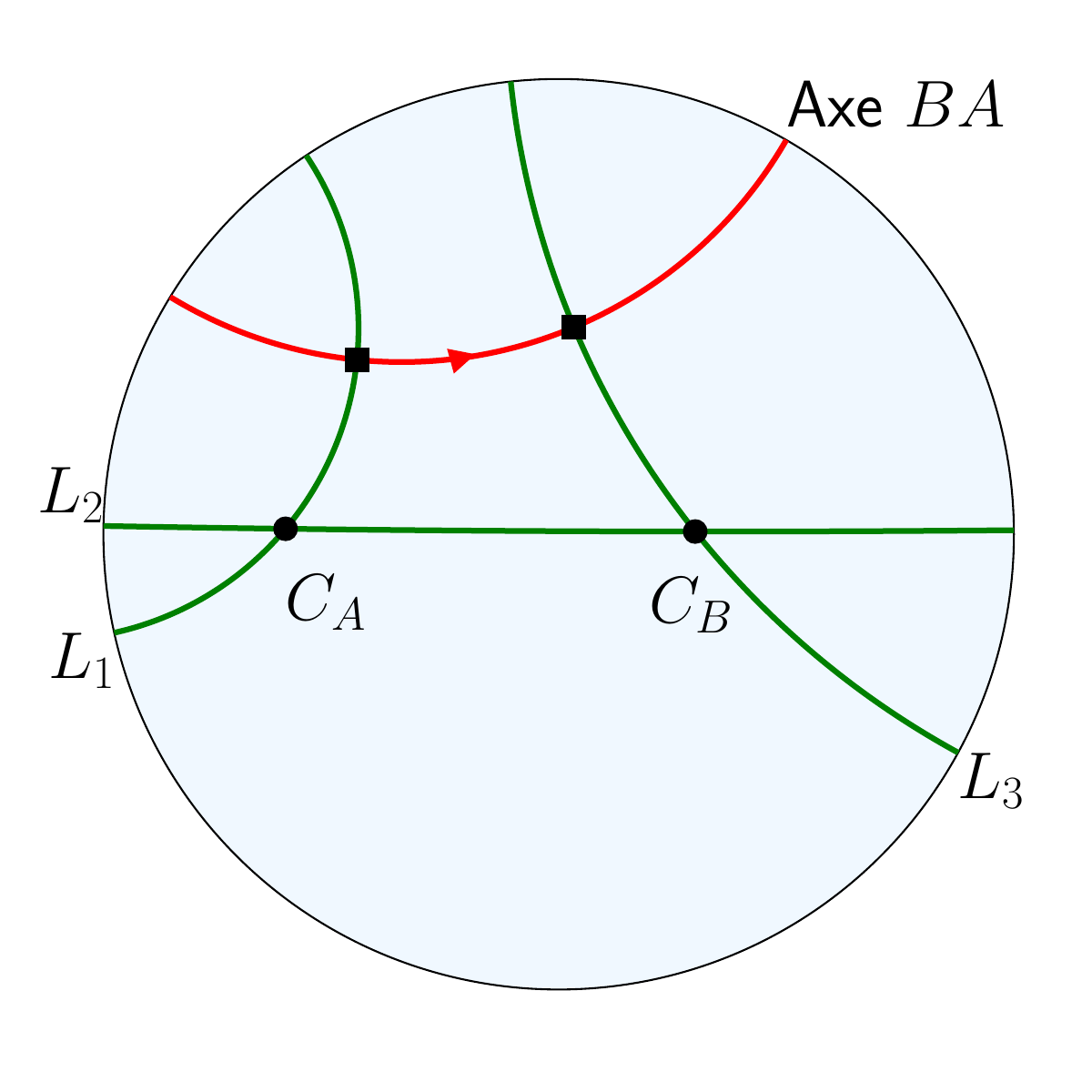}
		\caption{A product of two elliptic that is hyperbolic}
	\end{subfigure}
\end{figure}

\subsection{Product of a rotation and a translation}

In this paragraph, $A$ is a (hyperbolic) translation by $t>0$ along an axis $L$ and $B$ a rotation centred at $p$ of angle $\theta$. We denote by $q$ the orthogonal projection of $p$ onto $L$.

\vspace{2mm} We call $L_2$ the line through $p$ and $q$ and define $L_1$ as the line orthogonal to $L$ such that 

$$ A = \mathcal{R}_{L_1} \circ \mathcal{R}_{L_2} $$ and similarly we call $L_3$ the line through $p$ such that 

$$ B =  \mathcal{R}_{L_2} \circ \mathcal{R}_{L_3}.$$

\begin{lemma}
\label{lem:prodHR}
Let $A$ and $B$ be as above. There exists an open interval $I(t,d) = (\theta_{min}(t,d), \theta_{max}(t,d)) \subset S^1$ such that 

\begin{enumerate}
\item if $\theta \in I(t,d)$, $A\cdot B$ is elliptic;

\item if $\theta \in \{ \theta_{min}(t,d), \theta_{max}(t,d) \}$, $A\cdot B$ is parabolic;

\item if $\theta \in S^1 \setminus \overline{I}$, $A\cdot B$ is hyperbolic, in which case the pair $(A,AB)$ is of type $(\mathcal{H}, \mathcal{H})^+$.

\end{enumerate}

\end{lemma}

\begin{proof}
The proof is very similar to that of Lemma \ref{lem:prodRR}. After having written $AB = \mathcal{R}_{L_1} \circ\mathcal{R}_{L_3}$, we see that the type of $AB$ is determined by whether $L_1$ and $L_3$ intersect or not. When $\theta$ varies, $L_3$ runs over all lines through $p$. We thus see that there is an open interval $I \subset \mathbb{R}/2\pi \mathbb{Z}$ of angles for which $L_1$ and $L_3$ intersect. The extremities of $I$ correspond to $L_1$ and $L_3$ intersecting in $\partial \mathbb{H}$ and finally if $\theta \in  \mathbb{R}/2\pi \mathbb{Z} \setminus \overline{I}$, $L_1 \cap L_3 = \emptyset$. The Lemma then follows from Lemma \ref{lem:decomptype}.

\vspace{2mm} Since $A  = \mathcal{R}_{L_1} \circ \mathcal{R}_{L_2} $ and 
$AB = \mathcal{R}_{L_1} \circ\mathcal{R}_{L_3}$, both repelling points of $A$
and $AB$ are on the same of one of the two arcs of $\partial \mathbb{H}$ defined by the end points of $L_1$, whereas both repelling points are on the other. This implies that the pair $(A,AB)$ is of type $(\mathcal{H}, \mathcal{H})^+$.
\end{proof}

\begin{figure}[h]
	  \centering
	\begin{subfigure}{0.4\textwidth}
		  \centering
		\includegraphics[scale=0.3]{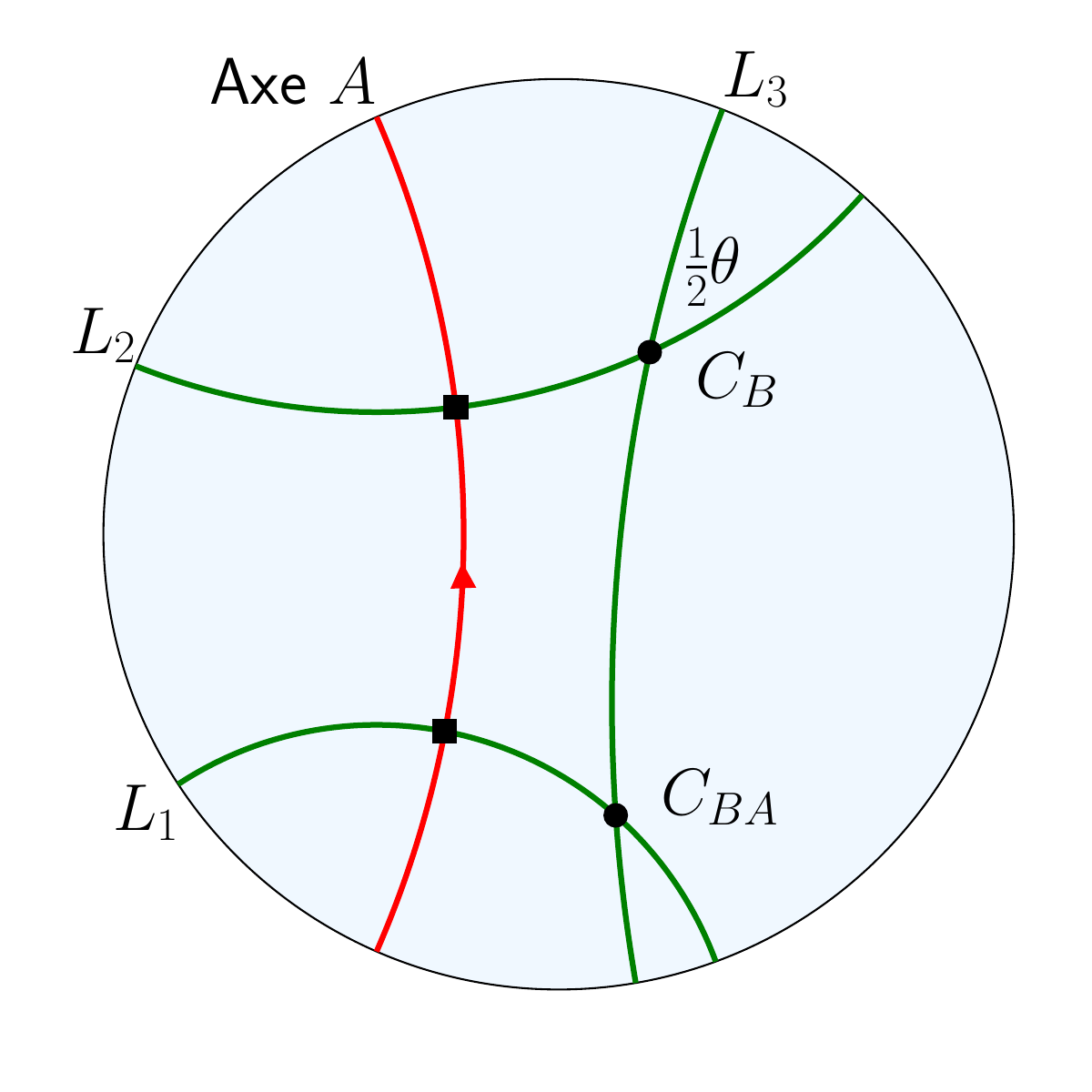}
		\caption{A product of a hyperbolic and a elliptic that is
		elliptic}
	\end{subfigure}
	\begin{subfigure}{0.4\textwidth}
		  \centering
		\includegraphics[scale=0.3]{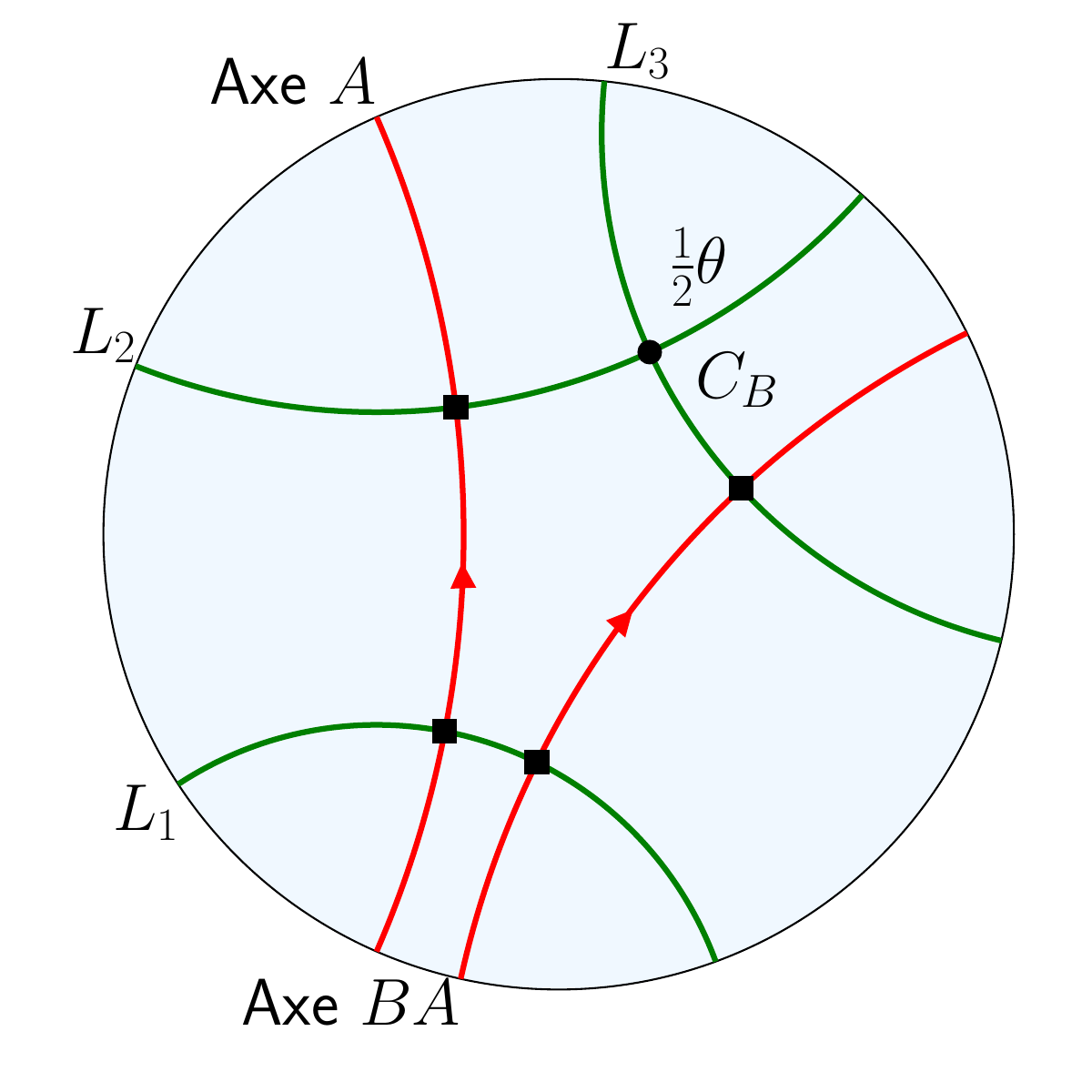}
		\caption{A product of a hyperbolic and a elliptic that is hyperbolic}
	\end{subfigure}
	\caption{Product of a hyperbolic and an elliptic}
	\label{fig:prod_HE}
\end{figure}

\subsection{Product of two translations}
In this paragraph we assume that both $A$ and $B$ are of hyperbolic type. Precisely

\begin{itemize}
\item $A$ is hyperbolic of axis $L_A$ and translation length $t_A$;

\item $B$ is hyperbolic of axis $L_B$ and translation length $t_B$;

\item we assume that $L_A \cap L_B = \emptyset$ and we denote by $L_2$ the line intersecting both $L_A$ and $L_B$ orthogonally.
\end{itemize}

Under this set of hypotheses, we have two potential configurations depending on the cyclic ordering of the fixed points of $A$ and $B$ on $\partial \mathbb{H}$. Either the attracting fixed points are consecutive on $\partial \mathbb{H}$ or attracting and repelling fixed points alternate. In what follows we make the assumption that we are in the latter case. In the terminology introduced earlier, this implies that the pair $(A,B)$ is of type $(\mathcal{H} \times \mathcal{H})^-$.

\begin{figure}[h]
	  \centering
		\includegraphics[scale=0.3]{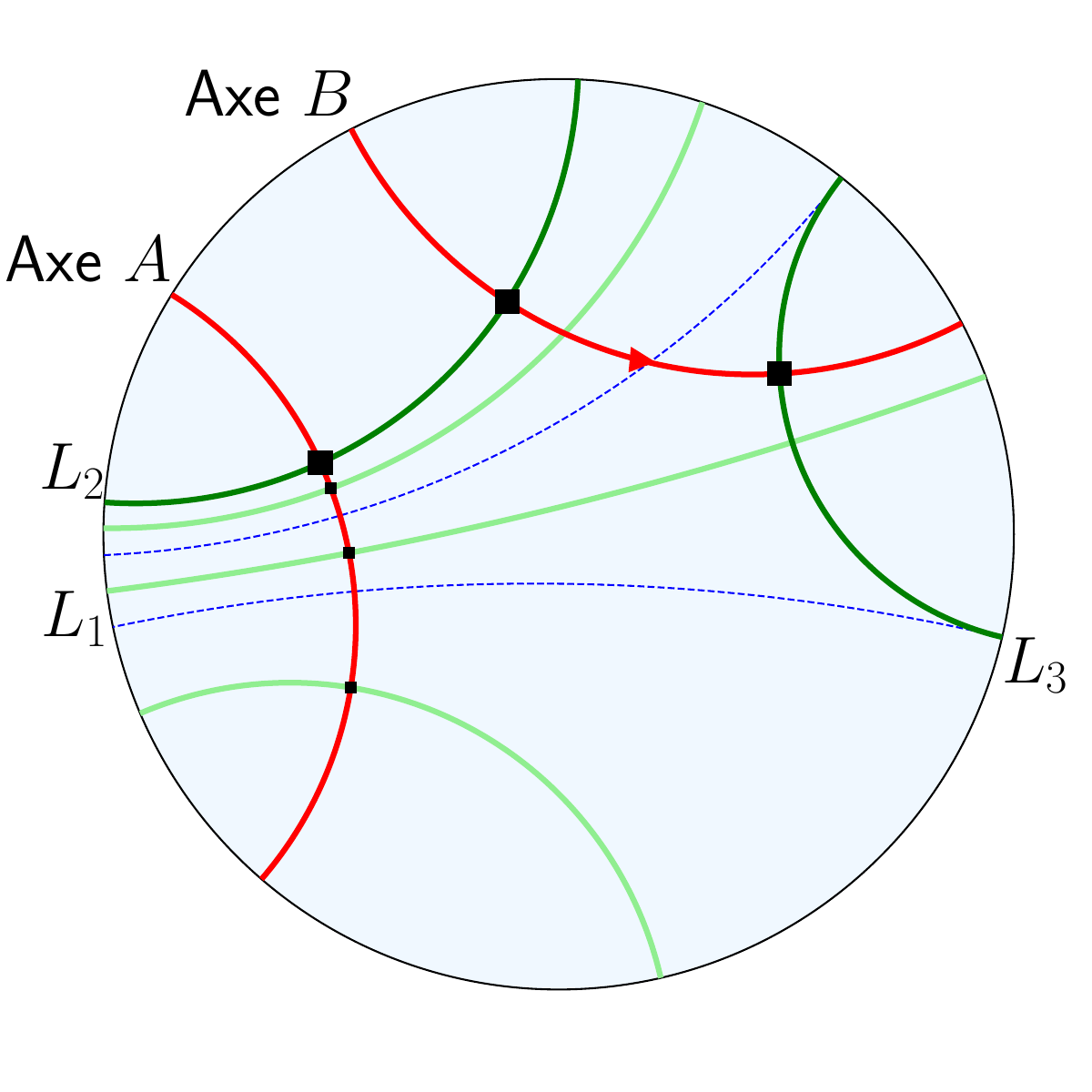}
		\caption{The three different cases for $ (H,H)^{-} $ }
\end{figure}

In this case, we call $L_1$ the line orthogonal to $L_A$ such that $$A = \mathcal{R}_{L_1} \circ \mathcal{R}_{L_2}$$ and $L_3$ the line such that 

$$ B = \mathcal{R}_{L_2} \circ \mathcal{R}_{L_3} $$ so that $A \cdot B = \mathcal{R}_{L_1} \circ \mathcal{R}_{L_3}.$

\begin{lemma}
\label{lem:prodHH}
Let $A$ and $B$ be a pair of matrices of type $(\mathcal{H}, \mathcal{H})^-$, and denote by $t_A, t_B \in \mathbb{R}^*_+$ the respective translation lengths of $A$ and $B$, and $d$ the distance from $L_A$ to $L_B$. There exists $t_1 = t_1(t_B,d)$ and $t_2 = t_2(t_B,d) >0$ such that the following holds

\begin{enumerate}
\item if $t_A \in (0, t_1)$, then $AB$ is hyperbolic and the pair $(A,AB)$ is of type $(\mathcal{H}, \mathcal{H})^-$;
\item if $t_A \in (t_1, t_2)$, then $AB$ is elliptic;
\item if $t_A \in (t_2, +\infty)$, then $AB$ is hyperbolic and the pair $(A,AB)$ is of type $(\mathcal{H}, \mathcal{H})^+$.
\end{enumerate}

\end{lemma}

\begin{proof}

Recall that $A = \mathcal{R}_{L_1} \circ \mathcal{R}_{L_2}$ and $A \cdot B =
\mathcal{R}_{L_1} \circ \mathcal{R}_{L_3}$. $L_1$ is a line that intersects
$L_A$ in a point which is at distance $\frac{t_A}{2}$ to $L_2$, so when $t_A$ is small enough the three lines $L_1, L_2, L_3$ do not intersect.
We thus see that the attracting fixed points of $A$ and the repelling one of $B$
(on $\partial \mathbb{H}$) are on the same side of $L_1$ and conversely that the attracting fixed points of $B$ and the repelling one of $A$ are on the other side. This implies that $(A,AB)$ is of type $(\mathcal{H}, \mathcal{H})^-$ if and only if their axes do not intersect. This is indeed the case since $L_1$ is orthogonal to both (and two distinct lines perpendicular to a same third line cannot intersect in $\mathbb{H}$).

\begin{figure}[h]
	  \centering
		\includegraphics[scale=0.3]{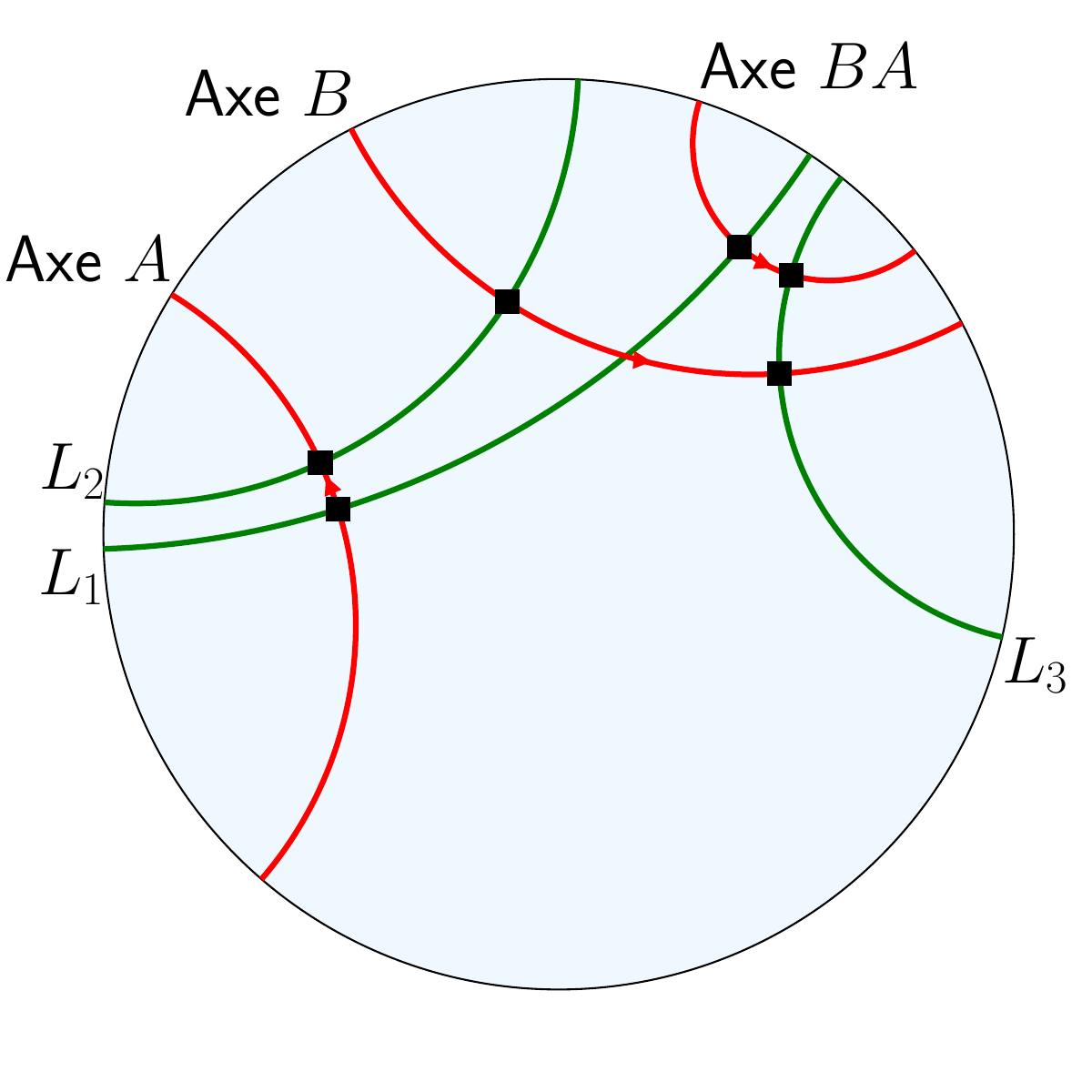}
		\caption{$(\Hyp,\Hyp)^{-}$: First case}
\end{figure}

\vspace{2mm} The configuration described above happens for all values of $t_A$ less than a certain constant $t_1$ for which $L_1$ shares an end point with $L_3$. For values of $t_A$ larger than $t_1$ but less than a certain constant $t_2$, $L_1$ and $L_3$ intersect in which case $AB$ is elliptic.

\begin{figure}[h]
	  \centering
		\includegraphics[scale=0.3]{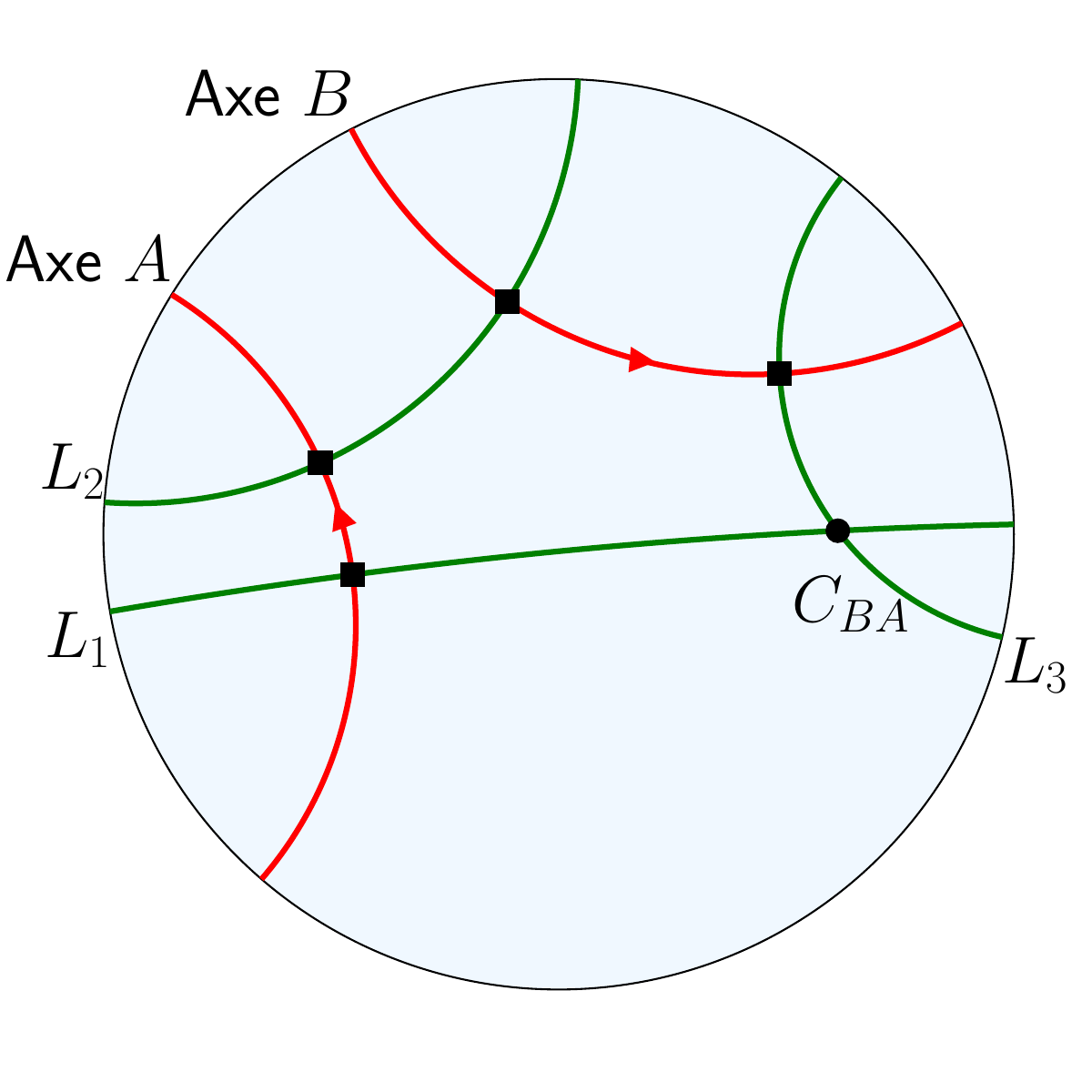}
		\caption{$(\Hyp, \Hyp)^{-}$: Second case}
\end{figure}

\vspace{2mm} Finally, for $t_A > t_2$, $L_1$ and $L_3$ no longer intersect. A
reasoning similar to that of the case $t_1 < t_A$, for the configuration shown in the Figure below, shows that $(A,AB)$ is of type $(\mathcal{H}, \mathcal{H})^+$. 

\begin{figure}[h]
	\centering
	\includegraphics[scale=0.3]{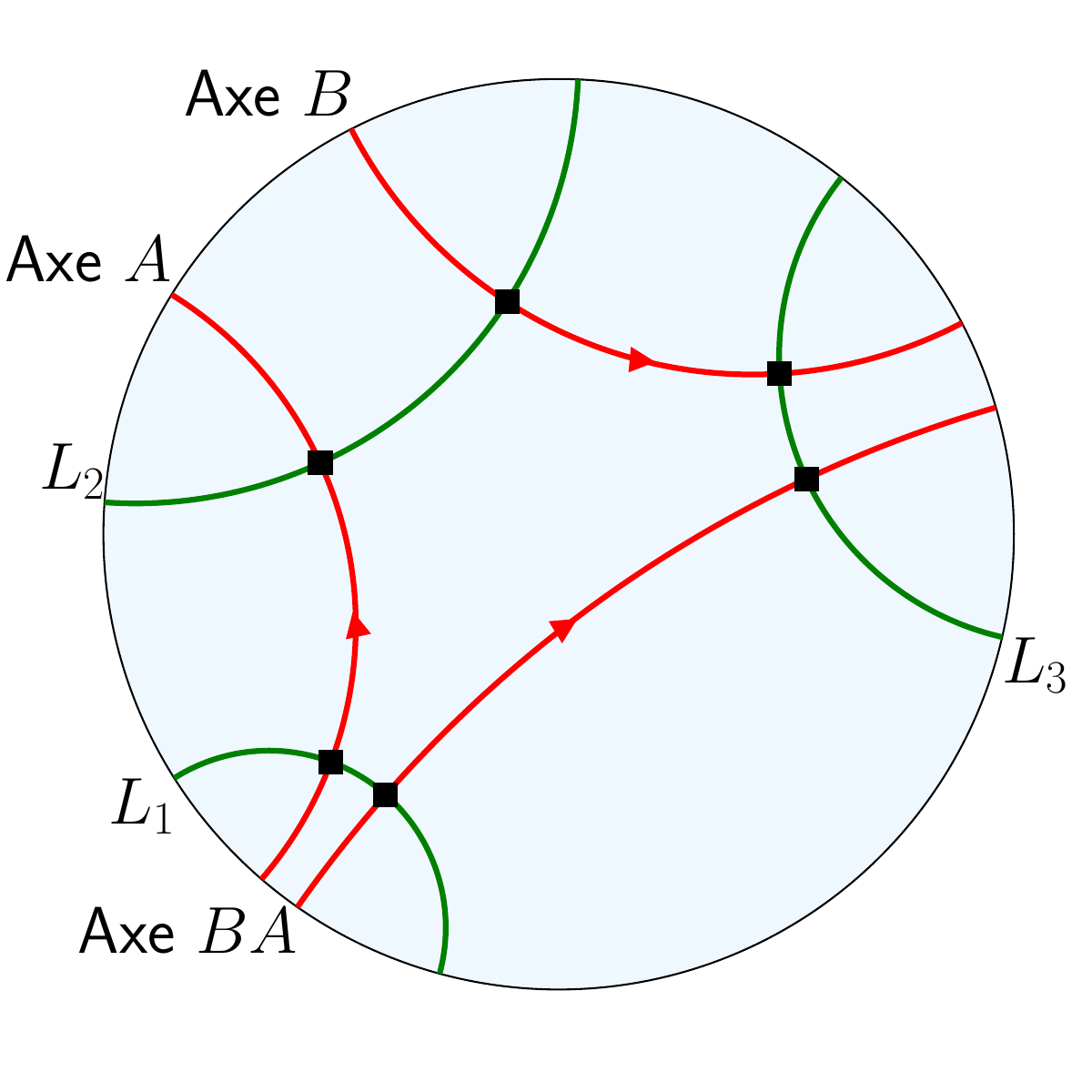}
	\caption{$(\Hyp, \Hyp)^{-}$: third case}\label{fig:fucshian}
\end{figure}

\end{proof}

\section{Technical lemmata}\label{sec:tech_lem}

\fbox{
\parbox{\textwidth}{ 
We recall that we have made the standing assumption that the matrices we are
dealing with \textit{and} the group they generate do not contain any parabolic
elements or a rotation of angle a rational multiple of $2\pi$.
}}

\vspace{2mm}

Based upon the lemmata from the previous section, we characterise the possible
ways in which a pair $(A,B)$ transforms under renormalization. We recall the
following notation, $\mathcal{E} := \{ M \in \mathrm{PSL}(2, \mathbb{R}) \ | \ M
\ \text{is elliptic} \}$ and $\mathcal{H} := \{ M \in \mathrm{PSL}(2, \mathbb{R}) \ | \ M \ \text{is hyperbolic} \}$.

\vspace{2mm} A step of renormalization consists formally in replacing a pair of
matrices $(A,B)$ with $(A, A^nB)$ (if renormalizing on the left) or $(A B^n, B)$
(if renormalizing on the right). Since these two processes are completely symmetrical, we will only concern ourselves with how the pair $(A,B)$ transforms under 

$$ (A, B) \longmapsto (A, A^n B) $$ and leave it to the reader to work out corresponding statements for the case $ (A, B) \longmapsto (AB^n,B)$.

\subsection{From $(\Ell,  \Ell)$.}

In this paragraph, we prove the following informal statement.

\begin{center}

If we start from a pair $(A,B) \in \mathcal{E} \times \mathcal{E}$ for which we do one step of accelerated renormalization, then we land infinitely many times on both configurations $( \mathcal{E}, \mathcal{E})$ and $ (\mathcal{E},\mathcal{H})$.
\end{center}

\begin{proposition}
\label{prop:EE}

Assume $(A,B) \in \mathcal{E} \times \mathcal{E}$ and that the respective angles of $A$ and $B$ are irrational multiples of $2 \pi$. Then 

\begin{enumerate}
	\item there are infinitely many $n \in \mathbb{N}$ such that $(A, BA^n) \in  \mathcal{E} \times \mathcal{E}$;
	\item there are infinitely many $n \in \mathbb{N}$ such that  $(A, BA^n) \in \mathcal{E} \times \mathcal{H}$.
\end{enumerate}

\end{proposition}
\begin{proof}

Since $A$ is a rotation of angle $\theta_A$, $A^n$ is a rotation of angle $n \cdot \theta_A \in \mathbb{R}/ 2\pi \mathbb{Z}$. Since $\theta_A$ is assumed to be an irrational multiple of $2\pi$, the sequence $n \cdot \theta_A$ is dense in $\mathbb{R}/ 2\pi \mathbb{Z}$. 

\vspace{2mm} With the notation of Lemma \ref{lem:prodRR}, we see that there are infinitely many $n$ such that $n \cdot \theta_A \in (-\alpha(d,\theta_B), \alpha(d,\theta_B))$ and infinitely many in $(\alpha(d,\theta_B), 2\pi - (-\alpha(d,\theta_B))$. For these $n$s, applying Lemma \ref{lem:prodRR} we get that $BA^n$ is a rotation or hyperbolic respectively. This proves Proposition \ref{prop:EE}.

\end{proof}

\subsection{From $(\mathcal{E}, \mathcal{H})$.}

In this paragraph, we prove the following informal statement.

\begin{center}

If we start from a pair $(A,B) \in \mathcal{E} \times \mathcal{H}$ for which we do one step of accelerated renormalization, then we land infinitely many times on both configurations $( \mathcal{E}, \mathcal{E})$ and $ (\mathcal{E},\mathcal{H})$.
\end{center}

\begin{proposition}

Assume $(A,B) \in \mathcal{E} \times \mathcal{H}$ and that the angle of $A$ is an irrational multiple of $2 \pi$. Then 

\begin{enumerate}
\item there are infinitely many $n \in \mathbb{N}$ such that $(A, BA^n) \in  \mathcal{E} \times \mathcal{E}$;

\item there are infinitely many $n \in \mathbb{N}$ such that  $(A, BA^n) \in \mathcal{E} \times \mathcal{H}$.
\end{enumerate}

\end{proposition}

\begin{proof}

The proof is very similar to that of Proposition \ref{prop:EE}. Since $A$ is a rotation of angle $\theta_A$, $A^n$ is a rotation of angle $n \cdot \theta_A \in \mathbb{R}/ 2\pi \mathbb{Z}$. Since $\theta_A$ is assumed to be an irrational multiple of $2\pi$, the sequence $n \cdot \theta_A$ is dense in $\mathbb{R}/ 2\pi \mathbb{Z}$. 

\vspace{2mm} With the notation of Lemma \ref{lem:prodHR}, we see that there are infinitely many $n$ such that $n \cdot \theta_A \in I(t,d)$ and infinitely many in $S^1 \setminus \overline{I(t,d)}$, where $t$ is the translation length of $B$ and $d$ the distance between the centre of $A$ and the axis of $B$. For these $n$s, applying Lemma \ref{lem:prodHR} we get that $BA^n$ is a rotation or hyperbolic respectively. This proves Proposition \ref{prop:EE}.

\end{proof}

\subsection{From $(\mathcal{H}, \mathcal{E})$.}

In this paragraph, we prove the following informal statement.

\begin{center}

If we start from a pair $(A,B) \in \mathcal{H} \times \mathcal{E}$ for which we do one step of accelerated renormalization, then we land \begin{itemize}
\item finitely many times in $( \mathcal{H}, \mathcal{E})$ 
\item infinitely many times in $ (\mathcal{H},\mathcal{H})^+$.
\end{itemize}

\end{center}

\begin{proposition}\label{prop:from_HE}
	Assume $(A,B) \in \mathcal{H} \times \mathcal{E}$. Then there exists $N = N(A,B) \in \mathbb{N}$
	\begin{enumerate}
		\item for all $n < N(A,B)$, $(A, BA^n) \in  \mathcal{H} \times \mathcal{E}$;

		\item for all $n \geq N(A,B)$, $(A, BA^n) \in  (\mathcal{H} \times \mathcal{H})^+$;
	\end{enumerate}
\end{proposition}

\begin{proof}
We need to understand the type of the $BA^n$ for arbitrary $n \in \mathbb{N}$. Using the same notation as in the proof of Lemma \ref{lem:prodHR}, we have $A = \mathcal{R}_{L_1} \circ \mathcal{R}_{L_2}$ and $AB = \mathcal{R}_{L_2} \circ \mathcal{L_3}$ where $L_2$ is the line through the centre of $B$ which is perpendicular to the axis of $A$. $A^n$ can be written $\mathcal{R}_{L_1^n} \circ \mathcal{R}_{L_2}$ where $L_1^n$ is a sequence of lines that are orthogonal to the axis of $A$, whose intersection point is an order sequence along this axis.

\vspace{2mm} \noindent This way one sees that if $L_1^{n_0}$ does not intersect
$L_3$, neither do all the $L_1^{n}$ for $n \geq n_0$. Define $N(A,B)$ the
smallest such $n_0$, which is easily be seen to exist. In the case where
$L_1^{n}$ does not intersect $L_3$, since $A = \mathcal{R}_{L_1} \circ
\mathcal{R}_{L_2}$ and $A^n B = \mathcal{R}_{L_1^n} \circ \mathcal{R}_{L_3}$,
and that the lines $L_1^n, L_1, L_2$ and $L_3$ are in the configuration shown in
Figure~\ref{fig:prod_HE}, the pair $(A,A^nB)$ is of type $(\mathcal{H}, \mathcal{H})^+$.

\end{proof}

\subsection{From $(\mathcal{H}, \mathcal{H})^{-}$.}
In this paragraph, we prove the following informal statement.

\begin{center}

If we start from a pair $(A,B) \in (\mathcal{H} \times \mathcal{H})^-$ for which we do one step of accelerated renormalization, then we land \begin{itemize}
\item finitely many times on $( \mathcal{H}, \mathcal{H})^-$;
\item finitely many times on $( \mathcal{H}, \mathcal{E})$;
\item infinitely many times on $ (\mathcal{H},\mathcal{H})^+$.
\end{itemize}

\end{center}

\begin{proposition}

Assume $(A,B) \in  (\mathcal{H} \times \mathcal{H})^-$. Then there exists $N_1 = N_1(A,B)$ and $N_2 = N_2(A,B) \in \mathbb{N}$, $N_1 \leq N_2$ such that 

\begin{enumerate}
\item if $n < N_1$, then $(A, BA^n) \in  (\mathcal{H} \times \mathcal{H})^-$;

\item if $N_1 \leq n < N_2$, then $(A, BA^n) \in  \mathcal{H} \times \mathcal{E}$;

\item if $n > N_2$, then $(A, BA^n) \in  (\mathcal{H} \times \mathcal{H})^+$.
\end{enumerate}

\end{proposition}

\begin{proof}

First note that, for $n >0$, the pair $(A, BA^n)$ is of type $(\mathcal{H}, \mathcal{H})^+$, $(\mathcal{H}, \mathcal{H})^-$ or $(\mathcal{H}, \mathcal{E})$ if and only if the pair $(A^n, BA^n)$ is of type $(\mathcal{H}, \mathcal{H})^+$, $(\mathcal{H}, \mathcal{H})^-$ or $(\mathcal{H}, \mathcal{E})$ respectively.

\vspace{2mm}\noindent The Proposition is thus a direct consequence of Lemma \ref{lem:prodHH} coupled with the fact the the translation length of $A^n$ is an increasing function of $n$ which tends to infinity.

\end{proof}

\subsection{Uniform hyperbolicity for elements $(\mathcal{H}, \mathcal{H})^{+}$.}

\begin{proposition}\label{prop:uh}
Assume $(A,B) \in  (\mathcal{H} \times \mathcal{H})^+$. Then there exist constants $\mu = \mu(A,B) > 1$ and $C = C(A,B) > 0$ such that the following holds. For any $2$-IET $T$, if $\rho$ is the cocycle associated to $(A,B)$, for all $x$ and for all $n \in \mathbb{N}$ we have 
$$ ||\rho(T^{n-1}(x)) \cdots \rho(T^{2}(x)) \rho(T(x)) \rho(x) || \geq C \cdot \mu^n.$$
and
$$ r \left( 
\rho(T^{n-1}(x)) \cdots \rho(T^{2}(x)) \rho(T(x)) \rho(x)
\right)
\geq C \cdot \mu^n,$$
where $ r $ is the spectral radius of a matrix.
\end{proposition}

\begin{proof}
The condition that $(A,B) \in  (\mathcal{H} \times \mathcal{H})^+$ implies the
existence of an interval $I$ in $\mathbb{RP}^1$ comprising both attracting fixed
points of $A$ and $B$, but not containing any of their repelling points.
This implies that the cone $\mathcal{C}_I$ in $\mathbb{R}^2$ defined as the set of lines whose associated direction in $\mathbb{RP}^1$ is in $I$ is 

\begin{itemize}

\item stable under both $A$ and $B$;

\item there exists a scalar product $|| \cdot ||_1$ on $\mathbb{R}^2$ such that
	the restriction of both $A$ and $B$ to $\mathcal{C}_I$  is uniformly
	dilating, by some factor $\mu(A,B)$.
\item each matrix in the semigroup generated by $ A $ and $ B $ is hyperbolic,
	and its attracting fixed point is in $ I $.
\end{itemize} 
This is the classical \emph{cone-criterion} for uniform hyperbolicity,
see~\cite{avilaUniformlyHyperbolicFiniteValued2010}.

The proof of the proposition follows from the fact that 

$$ ||\rho(T^{n-1}(x)) \cdots \rho(T^{2}(x)) \rho(T(x)) \rho(x) \vec{v} ||_1 \geq \mu(A,B)^n || \vec{v} ||_1 $$ for any vector $\vec{v} \in \mathcal{C}_I$ and the equivalence of norms on finite dimensional vector spaces.

\end{proof}

\section{Consequences for the Teichmüller flow acting on the character variety
bundle.}\label{sec:flow}

We now explain how our result can be reinterpreted in terms of an extension of the Teichmüller flow introduced by Goldman and Forni in \cite{ForniGoldman} for compact character varieties. We only give sketches of proofs as it is not the main point of our article.

\subsection{Moduli space of flat bundles} In this paragraph we briefly recall
the construction of the ``moduli space of flat bundles'' of Forni and Goldman. 

\vspace{2mm} Recall that the Teichmüller space of translation tori with a marked point is naturally isomorphic to $\mathrm{PSL}(2,\mathbb{R})$ (which one can see as the unit tangent bundle of the hyperbolic plane). The \textit{moduli} space of such structures is precisely the quotient of the unit tangent bundle by the action of the mapping class group of the torus with one marked point $\mathrm{SL}(2,\mathbb{Z})$ by isometries on $T^1 \mathbb{H}$.

\vspace{2mm} Now consider $\cv_c(\mathrm{T}^*, \mathrm{SL}(2,\mathbb{R}))$ a relative character variety, one can define 

$$ \mathcal{M}_c := \big( T^1 \mathbb{H} \times \cv_c(\mathrm{T}^*,
\mathrm{SL}(2,\mathbb{R})) \big)/ (z, [\rho]) \sim (A \cdot z, [\rho \circ
A^{-1}]).$$ The space $ \mathcal{M}_c $ is a $\cv_c(\mathrm{T}^*,
\mathrm{SL}(2,\mathbb{R}))$-bundle over $\mathrm{PSL}(2,\mathbb{R})/
\mathrm{PSL}(2,\mathbb{Z})$ the (unit tangent bundle of the) modular surface,
endowed with a flat connection inherited from that of the canonical product
structure of $ T^1 \mathbb{H} \times \cv_c(\mathrm{T}^*,
\mathrm{SL}(2,\mathbb{R}))$ which descends to the quotient.

\vspace{2mm} The geodesic flow on the modular surface lifts to $\mathcal{M}_c$ to define a lifted flow 

$$ (\Phi_t) : \mathcal{M}_c \longrightarrow \mathcal{M}_c $$ obtained via direct
parallel transport: if $(z, [\rho])$ is a point in $\mathcal{M}_c$, the orbit
$(\Phi_t(z, [\rho]))_{t \in \mathbb{R}}$ is the only flat section of the bundle $\mathcal{M}_c \longrightarrow \mathrm{PSL}(2,\mathbb{R})/ \mathrm{PSL}(2,\mathbb{Z})$ which projects onto the geodesic through $z$. For more detail about the construction we refer to \cite{ForniGoldman}.

\subsection{Divergence of the flow}

Our results (Theorems~\ref{thm:repr_spectrum}and~\ref{thm:repr_mcg}) have the
following corollary. Recall that an orbit of a flow is said to be divergent if it eventually leaves all compact sets. 

\begin{theorem}
The set of divergent orbits for the lift of the geodesic flow $ (\Phi_t) : \mathcal{M}_c \longrightarrow \mathcal{M}_c $ is an open and dense set. Its complement is exactly the set of recurrent orbits, and it is a lamination which is transversally a Cantor set. 
\end{theorem}
	
\paragraph*{\bf Sketch of proof.} Since the proof of this theorem is mostly about translating our results about renormalization of cocycles and the action of the mapping class group on particular representations in the language of Teichmüller dynamics, we only sketch how the process whereby one passes from one world to the other works. We require of our reader knowledge of the dynamics of the geodesic flow on the modular surface; we suspect the proof would be difficult to understand without prior knowledge of some facts about the latter. 

\vspace{2mm} \noindent The geodesic flow on $\mathrm{PSL}(2,\mathbb{R})/ \mathrm{PSL}(2,\mathbb{Z})$ has the property that a geodesic either diverges in the cusp or comes back infinitely many times to an open set $U$ which projects onto a simply connected subset of $\mathbb{H}/\mathrm{PSL}(2,\mathbb{Z})$. In particular the flat bundle $\mathcal{M}_c$ is trivial over $U$, and therefore identifies in a canonical way to 

$$ U \times \cv_c(\mathrm{T}^*, \mathrm{SL}(2,\mathbb{R})).$$ The set of trajectories that escape in the cusp is a countable union of $2$-dimensional submanifolds of $\mathrm{PSL}(2,\mathbb{R})/ \mathrm{PSL}(2,\mathbb{Z})$ and we will therefore dismiss them. 

\vspace{2mm} \noindent Take a point $x \in \mathbb{U}$, and let $t_n$ be an
increasing sequence of times $(t_n)$ tending to infinity such that  $g_{t_n}(x)
\in U$ (we think of $(t_n)$ as a sequence of return times to $U$ after possible
excursions in the cusp). Closing the geodesic from $x$ to $g_{t_n}(x)$ within
$U$ produces a sequence of elements $A_n$ in $\mathrm{SL}(2,\mathbb{Z})$, which
has the property that if we started from $z = (x,[\rho]) \in U \times \cv_c(\mathrm{T}^*, \mathrm{SL}(2,\mathbb{R}))$, we have 

$$ \Phi_{t_n}(z) = (g_{t_n}(x), [\rho \circ A_n]).$$ Furthermore, the sequence $A_n$ is precisely the sequence given by the \textbf{accelerated} renormalization of the cocycle associated to $\rho$ for the measured foliation corresponding to the geodesic through $x$. Although classical, this part would probably require more rigorous justification. It will only hold true for a careful choice of $U$ and careful choice of cross section of the foliation associated to $x$.

\vspace{2mm} If we had chosen $\rho$ to be a generic representation, we would have obtained that for a choice of $x$ in an open and dense set, the iterated renormalization of the associated IET would eventually land in $(\mathcal{H}, \mathcal{H})^+$.
From that point onwards,
by Theorem~\ref{thm:repr_mcg},
renormalizations of the associated cocycle
$ (A_n,B_n) $
would grow like
$ e^{C q_n} $.
Appealing again to our reader's knowledge of the connection between the geodesic flow on  $\mathrm{PSL}(2,\mathbb{R})/ \mathrm{PSL}(2,\mathbb{Z})$ and continued fraction, we have that $q_n$ is equivalent to $t_n$. Now the condition that renormalization of a cocycle land in $(\mathcal{H}, \mathcal{H})^+$ after $n$ renormalizations is an open condition in the representation; so we have proved the following statement:

\begin{center}
	\it
	There is an open and dense set $V$ of $\mathcal{M}_c$ such that if $z \in V$,
	$\Phi_t(z)$ tends to infinity as $t$ tends to infinity. Furthermore, if the
	projection of $(\Phi_t(z))$ is recurrent in the modular surface, the
	``projection'' of $\Phi_t(z)$ onto the fibre $\cv_c(\mathrm{T}^*, \mathrm{SL}(2,\mathbb{R}))$.
\end{center}

The latter projection onto $\cv_c(\mathrm{T}^*, \mathrm{SL}(2,\mathbb{R}))$ only makes sense over $U$, where it is defined using the trivialisation coming from the flat connection.

\vspace{2mm} Finally, let $F$ be the complement of $V$.
If $z = (x,\rho) \in F$ and $x$ corresponds to an irrational measured foliation
(\textit{i.e.} the orbit of $x$ under the geodesic flow is recurrent to $U$),
the iterated accelerated renormalizations of $\rho$ over the measured foliation
corresponding to $x$ remain bounded (by Theorem~\ref{thm:repr_mcg}), it will therefore be the
 case for the orbit $\Phi_t(z)$ when restricted to $U \times \cv_c(\mathrm{T}^*,
 \mathrm{SL}(2,\mathbb{R}))$. $F$ therefore plays the role of a ``trapped set''
 for the dynamics of $\Phi_t$, which is the analogue of the set of recurrent
 trajectories for the geodesic flow on a geometrically finite hyperbolic surface
 of infinite volume. The funnels of such a surface act as the subset $V$ of points whose trajectories contain a neighbourhood that escapes directly into the funnel. 

 \section{Open questions}\label{sec:open}

\subsection{Spectrum of representations}

For this series of questions, we consider $\Sigma$ a closed surface of genus $g \geq 2$ and $\rho : \pi_1 \Sigma \longrightarrow \mathrm{SL}(2,\mathbb{R})$ a non-elementary\footnote{A representation is called non-elementary if its image is not Zariski-dense in $\mathrm{SL}(2,\mathbb{R})$.} representation which we assume not to be Fuchsian. One can define its spectrum in the exact same way as we have done in this article: it is the set of measured laminations for which the cocycle associated to $\rho$ has vanishing Lyapunov exponent.

\begin{question}
What is the topological structure of the spectrum of $\rho$?
\end{question}

The spectrum is going to be a subset of the space of measured laminations, which is known to be a sphere of dimension $6g - 7$. It seems reasonable to conjecture that it is Cantor-like (probably it is locally the product of a Cantor set with a $(6g - 8)$-dimensional disk).

\vspace{2mm} It is known that the \textit{marked length spectrum} fully determines a Fuchsian representation. We suggest a rigidity problem analogous to this one for non-Fuchsian representation.
It is also closely related to the \emph{ending laminations conjecture} (now a theorem)
for Kleinian representations.
  
\begin{question}
If $\rho_1$ and $\rho_2$ are two non-simple Fuchsian representations into $\mathrm{SL}(2,\mathbb{R})$ that have the same spectrum (as a subset of the space of measured laminations), is it true that they are conjugate?
\end{question}

We have no particular reason to believe that it should be the case other than the fact the spectrum captures a lot of geometric information about a representation;  it would be very surprising if it were not sufficient to completely characterise it.

\vspace{2mm} We proved in the case of representations $\rho : \pi_1 \mathrm{T}^* \longrightarrow \mathrm{SL}(2,\mathbb{R})$ that the spectrum is a Cantor set. A general question that one can ask is that of its \textit{size}.

\begin{question}
When the spectrum of $\rho$ is a Cantor set, does it have zero measure? Can one compute its Hausdorff dimension?
\end{question}

\paragraph*{\it The $\mathrm{SL}(2,\mathbb{C})$-case.} The definition of the spectrum carries over without any changes to the case where the target group is $\mathrm{SL}(2,\mathbb{C})$, which is of interest since it corresponds to hyperbolic geometry in dimension $3$. 

\noindent We restrict our question to the case where $\Sigma$ is a closed
surface of genus $g \geq 2$, although it has a version with boundary which is a
bit more technical to state, and of which we spare our reader.  A non-elementary
representation $\rho : \pi_1 \Sigma \longrightarrow \mathrm{SL}(2,\mathbb{C})$
can be quasi-Fuchsian, in which case its spectrum is empty. The boundary of the
quasi-Fuchsian locus contains representations which can have a pair of transverse
lamination in their spectrum (their \textit{``ending laminations''}). 

\begin{question}
Let $\rho : \pi_1 \Sigma \longrightarrow \mathrm{SL}(2,\mathbb{C})$ be in the complement of the closure of the set of quasi-Fuchsian representation. Is it true that the spectrum of $\rho$ is a Cantor set?
\end{question}

\noindent Tentative numerical experiments in the case of the one-holed torus seem to suggest that a positive answer to this question is plausible.

\subsection{Cocycles over interval exchange transformations}

A fact that we build upon in this article is that studying the spectrum of a representation amounts to studying piecewise constant cocycles over interval exchange transformation. For the general case of  representations of higher genus surfaces, this fact is still going to hold true. This is a justification for the series of questions below, although we think they are interesting in their own right. 

\vspace{2mm} Indeed the study of piecewise constant cocycles with values in $\mathrm{SL}(2,\mathbb{R})$ provides a non-abelian extension of the classical ergodic theory of interval exchange maps. This is formally analogous to how the character variety bundle introduced by Forni and Goldman is a non-abelian analogue of the Kontsevich-Zorich cocycle.

\vspace{2mm} As such it would be reasonable to speculate that the theory of Higgs bundles could give a non-abelian analogue of the Hodge norm used by Forni in his study of the hyperbolicity of the Kontsevich-Zorich cocycle \cite{Forni}. However, it is unclear what sort of results to expect here, as KAM theory tells us that it cannot be fully non-uniformly hyperbolic (see \cite{FGLM}).

\vspace{2mm} \paragraph*{\it Measurable classification of
$\mathrm{SL}(2,\mathbb{R})$-cocycles.} We recall that measurable cocycles over an ergodic measurable dynamical system $(T,X,\mu)$ admit a fairly simple classification (see \cite{Thieullen}). A cocycle $A : X \longmapsto \mathrm{SL}(2,\mathbb{R})$ can be 

\begin{enumerate}
\item \textit{hyperbolic} meaning that its Lyapunov exponent is positive;
\item \textit{parabolic} meaning that it is measurably conjugate to a cocycle of the form $ x \mapsto \begin{pmatrix}
v(x) & m(x)  \\
0 & \frac{1}{v(x)}
\end{pmatrix}$ where $v : X \longrightarrow \mathbb{R}^*_+$ is a measurable function such that $\log v$ has zero average.
\item \textit{elliptic} meaning that it is measurably conjugate to a cocycle of the form $ x \mapsto \begin{pmatrix}
\cos \theta(x) & -\sin \theta(x) \\
\sin \theta(x) & \cos \theta(x)
\end{pmatrix}$ where $\theta : X \longrightarrow \mathbb{R}$ is a measurable function.
\end{enumerate}

\begin{question}

\begin{itemize}

\item Is there a pair $(T_0, \rho)$ where $T_0$ is an ergodic interval exchange transformation and $\rho$ a cocycle whose image in $\mathrm{SL}(2,\mathbb{R})$ is non-elementary which is of parabolic type?

\item Same question for elliptic type.

\item Is there such a pair which is hyperbolic but not uniformly hyperbolic?

\end{itemize}

\end{question}

Note that uniform hyperbolicity can be achieved by any choice of cocycle whose matrices satisfy a criterion analogous to the case $(\mathcal{H} \times \mathcal{H})^+$.

\vspace{2mm} \paragraph*{\it Renormalisation of cocycles.} In this article we have shown that the following statement is more or less correct. A piecewise constant cocycle over a rotation either has bounded renormalization and its Lyapunov exponent is equal to $0$, or it is uniformly hyperbolic and its renormalizations diverge to infinity at a (precise) exponential rate. 
\noindent This can be thought of as a non-abelian extension of the classical ergodic theory over rotations; and our statement is nothing but a (formal) analogue of the celebrated Denjoy-Koksma inequality. It is well-known that for IETs of four or more intervals, there is no equivalent of the Denjoy-Koksma inequality; and there are abelian cocycles with zero Lyapunov exponent for which renormalizations diverge at an exponential rate.

\begin{question}
Let $T_0$ be an ergodic IET on four or more intervals, with good Diophantine properties. Is there a non-elementary cocycle $\rho$ with zero Lyapunov exponent such that the iterated renormalizations of $\rho$ tend to infinity in the space of cocycles at exponential speed?
\end{question} 

This question is related to the existence of measurably reducible cocycles.

\vspace{2mm} \paragraph*{\it Varying the cocycle.} A natural question from the viewpoint of cocycle dynamics is to fix the base IET and vary the cocycle. There is a rich theory for smooth $\mathrm{SL}(2,\mathbb{R})$-cocycles over a fixed rotation (see \cite{AvilaFayadKrikorian,AvilaKrikorian} and references therein).

\begin{question}
Let $T_0$ be an ergodic IET on $d$ intervals with $d \geq 4$. In the space of
piecewise constant cocycles (which is $\mathrm{SL}(2,\mathbb{R})^d$), what does
the locus where the Lyapunov exponent vanishes look like?

\end{question}

\vspace{2mm} \noindent In the introduction, we expressed the hope that the
theory of ``abelian'' cocycles could help predict what happens in the non-abelian case. 
\noindent In the case of abelian cocycle, for a given (generic) IET, the space of cocycles admits a stratification reflecting a set of obstructions introduced by Forni \cite{Forni}. The locus where the Lyapunov exponent vanishes is the first non-trivial level of the stratification, and each of level correspond to a different speed at which a given cocycle diverges to infinity (positive Lyapunov exponent correspond to exponential divergence, intermediate strata to a divergence which is subexponential, and the lowest level corresponds to bounded cocycles. We ask the following vague question.

\begin{question}
Let $T_0$ be an ergodic IET on $d$ intervals with $d \geq 4$. In the space of piecewise constant cocycles (which is $\mathrm{SL}(2,\mathbb{R})^d$), can the locus where the Lyapunov exponent vanish be stratified according to Forni's obstruction as to reflect speed of divergence under renormalization?

\end{question}

\bibliographystyle{alpha}
\bibliography{biblio}

\end{document}